\documentclass[12pt]{amsart}
\usepackage{amsmath}
\usepackage{amssymb}
\usepackage{amsthm}
\usepackage[all]{xy}
\usepackage{graphicx}
\usepackage[margin=2.5cm]{geometry}
\usepackage{color}
\usepackage{hyperref}
\usepackage{verbatim}
\usepackage{pb-diagram,pb-xy}

\newcommand{\R}{\ensuremath{\mathbb{R}}}

\newcommand{\Z}{\ensuremath{\mathbb{Z}}}

\newtheorem{theorem}{Theorem}[section]

\newtheorem{lemma}[theorem]{Lemma}
\newtheorem{proposition}[theorem]{Proposition}

\newtheorem{corollary}[theorem]{Corollary}

\theoremstyle{definition} 
\newtheorem{defn}{Definition}[section]
 
\newtheorem{remark}{Remark}[section] 
\newtheorem{Notation}{Notational Convention}[section]
\newtheorem{Conjecture}{Conjecture}[section]

\DeclareMathOperator*{\BDiff}{{\mathbf B}Diff}

\DeclareMathOperator*{\Int}{int}

\DeclareMathOperator*{\Diff}{Diff}

\DeclareMathOperator*{\Emb}{Emb}

\DeclareMathOperator*{\Hom}{Hom}

\DeclareMathOperator*{\Ker}{Ker}
\DeclareMathOperator*{\Coker}{Coker}
\DeclareMathOperator*{\Torsion}{Torsion}
\DeclareMathOperator*{\link}{link}

\author{Nathan Perlmutter}

\address{Department of Mathematics, University of Oregon, Eugene, OR,
  97403, USA} 
  
  \email{nperlmut@uoregon.edu}

\title[Homological Stability For Moduli Spaces of Odd Dimensional
  Manifolds]{Homological Stability For Moduli Spaces of Odd
  Dimensional Manifolds}
  
\begin{document}
\maketitle
 \begin{abstract} 
We prove a homological stability theorem for the moduli spaces of
manifolds diffeomorphic to $\#^{g}(S^{n+1}\times S^{n})$, provided $n
\geq 4$. This is an odd dimensional analogue of a recent
homological stability result of S. Galatius and O. Randal-Williams for the
moduli spaces of manifolds diffeomorphic to $\#^{g}(S^{n}\times S^{n})$
for $n \geq 3$.
 \end{abstract}
\section{Introduction} \label{Introduction} 
\subsection{Main result}
  Recently S. Galatius and O. Randal-Williams proved a homological
  stability theorem for moduli spaces of certain manifolds of dimension
  $2n$ where $n \geq 3$, see \cite{GRW 12}.
For a positive integer $g$,
denote by $V_{g,1}$ the connected sum $\#^{g}(S^{n}\times S^{n}) -
\Int(D^{2n})$ and by $\Diff(V_{g,1})^{ \partial}$ the topological
group of diffeomorphisms of $V_{g,1}$ which fix a neighborhood of the boundary pointwise,
equipped with the $C^{\infty}$-topology. There is a natural map
$$s_{g}: \BDiff(V_{g,1})^{\partial} \longrightarrow
\BDiff(V_{g+1,1})^{\partial}$$ induced by extending a diffeomorphism
identically over an attached handle.  S. Galatius and
O. Randal-Williams proved that the map on homology induced by $s_{g}$
$$H_{k}(\BDiff(V_{g,1})^{\partial};\Z) \;
  \longrightarrow \; H_{k}(\BDiff(V_{g+1,1})^{\partial};\Z),$$ 
  is an isomorphism provided that $k
\leq \frac{1}{2}(g-4)$, see \cite[Theorem 1.2]{GRW 12}.

In this paper we prove an analogous result for diffeomorphism groups of similar odd-dimensional manifolds
of dimension at least nine.  

For a positive integer $g$, denote by $W_{g, 1}$ the
$(2n+1)$-dimensional manifold with boundary $\#^{g}(S^{n+1}\times
S^{n}) - \Int(D^{2n+1})$. Like with the case above, 
there is a natural
map
\begin{equation} \label{eq: stabilization map intro} 
s_{g}: \BDiff(W_{g,1})^{\partial} \longrightarrow
\BDiff(W_{g+1,1})^{\partial} 
\end{equation} 
induced by extending a diffeomorphism identically over an attached
copy of the cobordism $W_{1, 1} - \Int(D^{2n+1})$. This map is defined rigorously in Section \ref{The Moduli Spaces}. We can now state our main theorem.
\begin{theorem} \label{thm: Main Theorem} 
For $n \geq 4$ the map on homology induced by {\rm
    (\ref{eq: stabilization map intro})}
$$(s_{g})_{*}: H_{k}(\BDiff(W_{g,1})^{\partial}; \Z) \longrightarrow
  H_{k}(\BDiff(W_{g+1,1})^{\partial}; \Z),$$ is an isomorphism
  provided that $k \leq \frac{1}{2}(g-3)$.
\end{theorem}
\subsection{Ideas behind the proof}
Our methods are similar to those used in \cite{GRW 12}. Namely, we
construct a highly connected semi-simplicial space
$K_{\bullet}(W_{g,1})$ which admits a transitive action of the
topological group $\Diff(W_{g,1})^{\partial}$. Roughly, the
$p$-simplices of $K_{\bullet}(W_{g,1})$ are given by ordered lists of
$(p+1)$-many pair-wise disjoint embeddings of $W_{1, 1}$ into $W_{g,1}$
with a pre-prescribed boundary condition. This semi-simplicial
space is almost identical to the one constructed in \cite{GRW 12} and
we use many of the same formal results from \cite{GRW 12} regarding
this semi-simplicial space.  For instance, after proving that the
geometric realization $|K_{\bullet}(W_{g,1})|$ is highly connected,
the proof of Theorem \ref{thm: Main Theorem} goes through similarly to
the proof of \cite[Theorem 1.2]{GRW 12}.

However, in the proof of high-connectivity of
$|K_{\bullet}(W_{g,1})|$, our methods differ from those of \cite{GRW
  12}.  Our proof requires much of the theory regarding the
diffeomorphism classification of highly connected manifolds developed
by C.T.C. Wall in \cite{Wa 63} and \cite{Wa 66}, specialized to the
case of $(n-1)$-connected, $(2n+1)$-dimensional manifolds. 

In more detail, 
let $M$ be an $(n-1)$-connected, $(2n+1)$-dimensional manifold. 
The homotopy groups $\pi_{n}(W_{g,1})$ and $\pi_{n+1}(W_{g,1})$ come equipped with
the intersection parings
$$
\begin{array}{c}
\lambda_{n, n+1}: \pi_{n}(M)\times \pi_{n+1}(M)
\longrightarrow \pi_{n}(S^{n}), 
\\
\\
\lambda_{n+1, n+1}: \pi_{n+1}(M)\times \pi_{n+1}(M)
\longrightarrow \pi_{n+1}(S^{n}),
\end{array}
$$
and maps
$$
\alpha_{i}: \pi_{i}(M) \longrightarrow \pi_{i-1}(SO_{2n+1-i}) \quad \quad \text{for \; $i = n, \; n+1$}
$$
defined by sending a class $x \in \pi_{i}(M)$ to the
class in $\pi_{i-1}(SO_{2n+1-i})$ which classifies the normal
bundle of an embedding which represents $x$. (Indeed,
  for $n \geq 4$, the connectivity assumption on $M$ implies that all such elements $x \in \pi_{i}(M)$ for $i = n, n+1$ are represented by embeddings, see \cite{Ha 62}).  

C.T.C. Wall shows that $\lambda_{n+1, n+1}$
and $\alpha_{n+1}$ are related by the summation formula,
$$\alpha_{n+1}(a + b) \; = \; \alpha_{n+1}(a) \; + \; \alpha_{n+1}(b)
\; + \; \partial(\lambda_{n+1, n+1}(a, b)) \quad \text{for all $a, b
  \in \pi_{n+1}(M)$}
$$
where $\partial: \pi_{n+1}(S^{n}) \rightarrow \pi_{n}(SO_{n})$ is the
boundary map in the long-exact sequence associated to the fibre
sequence, $SO_{n} \rightarrow SO_{n+1} \rightarrow S^{n}$, see \cite{Wa 63}. 

For the case that $M = W_{g,1}$, the map $\alpha_{n}: \pi_{n}(W_{g,1}) \rightarrow \pi_{n-1}(SO_{n+1})$ is identically zero and 
the algebraic structure given by
\begin{equation} \label{eq: standard wall pairing} 
{\mathbb W}_{g}:= 
(\pi_{n}(W_{g,1}), \; \pi_{n+1}(W_{g,1}), \; \lambda_{n, n+1}, \; \lambda_{n+1, n+1}, \;  \alpha_{n+1}) 
\end{equation}
has a particularly nice form. 
In this case we have isomorphisms
$$
\pi_{n}(W_{g,1})
\cong \Z^{\oplus g}, \ \ \ 
\mbox{and} \ \ \ \pi_{n+1}(W_{g, 1}) \cong \Z^{\oplus
  g}\oplus (\Z/2)^{\oplus g}.
$$ 
Furthermore, for the groups $\pi_{n}(W_{g,1})$ and $\pi_{n+1}(W_{g,1})$
there exist bases 
$$(x_{1}, \dots, x_{g}) \quad \text{and} \quad (y_{1}, \dots, y_{g}, z_{1}, \dots, z_{g}),$$
with 
$(z_{1}, \dots, z_{g})$ a basis for the $\Z/2$-component of $\pi_{n+1}(W_{g,1})$, such that for $i, j \in \{1, \dots, g\}$, 
$$\alpha_{n+1}(y_{i}) = \alpha_{n+1}(z_{i}) = 0, \quad \lambda_{n,
  n+1}(x_{i}, y_{j}) = \delta_{i, j}\cdot\iota_{n}, \quad
\lambda_{n+1, n+1}(y_{i}, z_{j}) = \delta_{i, j}\cdot\rho_{n+1},$$
where $\iota_{n} \in \pi_{n}(S^{n})$ and $\rho_{n+1} \in
\pi_{n+1}(S^{n})$ are the standard generators. We call the algebraic
structure given by the $5$-tuple
${\mathbb W}_{g}$ from (\ref{eq: standard wall
  pairing}) a \textit{Wall pairing of rank $g$.} In order to prove
that the space $|K_{\bullet}(W_{g,1})|$ is highly connected we study
the action of $\Diff(W_{g,1})^{\partial}$ on the Wall
  pairing ${\mathbb W}_{g}$.

We construct a simplicial complex $K^{\pi}(W_{g,1})$
whose $p$-simplices are given by sets of $(p+1)$-many
pair-wise orthogonal (with respect to both $\lambda_{n, n+1}$ and
  $\lambda_{n+1, n+1}$) embeddings ${\mathbb W}_{1}\hookrightarrow
  {\mathbb W}_{g}$ of the Wall pairings, mimicking embeddings of the
  manifolds $W_{1,1}\hookrightarrow W_{g,1}$.

In this sense, the complex $K^{\pi}(W_{g,1})$ is an algebraic version
of the semi-simplicial space $K_{\bullet}(W_{g,1})$.  In Section
\ref{section: The Intersection Complex} we prove that the geometric
realization $|K^{\pi}(W_{g,1})|$ is $\frac{1}{2}(g-3)$-connected. The
connectivity of $|K^{\pi}(W_{g,1})|$ allows us to establish a lower
bound for the connectivity of $|K_{\bullet}(W_{g,1})|$.  At this point we use 
techniques developed by R. Charney \cite{Ch 87, Ch 84} and a theorem of W. Van der Kallen \cite{Ka 80}, see Theorem \ref{theorem: unimodular connectivity}.  In
particular, the proof that the geometric realization 
$|K^{\pi}(W_{g,1})|$ is $\frac{1}{2}(g-3)$-connected is similar to the
proof of Theorem 3.2 from \cite{Ch 87}.

One of the main geometric difficulties encountered is that we must
deal with intersections of embedded sub-manifolds above the middle
dimension.  In the case when $n \geq 4$, i.e. when $\dim(W_{g,1}) \geq
9$, we are in the right dimensional and connectivity range to apply
Haeffliger's results from \cite{Ha 62} to identify the homotopy group
$\pi_{n+1}(W_{g,1})$ with the set of isotopy classes of embedded
spheres. This allows us access to the machinery developed by
C.T.C. Wall in \cite{Wa 63} to handle these intersections of
sub-manifolds above the middle dimension.

For the case $n < 4$ not all classes of
$\pi_{n+1}(W_{g,1})$ are represented by embeddings. More care is
required to compute the intersections and self-intersections of the
classes of $\pi_{n+1}(W_{g,1})$. As a result, for $n=2,3$, the
relevant complex $K^{\pi}(W_{g,1})$ has a different structure than
what is described and studied in this paper.  However, the author
believes that Theorem \ref{thm: Main Theorem} holds also if $n=2,3$.
\subsection{A Conjecture} \label{A Conjecture} 
The maps $\lambda_{n, n+1}$, $\lambda_{n+1, n+1}$, $\alpha_{n}$, and
$\alpha_{n+1}$ described in the previous section are defined on the
homotopy groups of any $(n-1)$-connected, $(2n+1)$-dimensional
manifold. It is then natural to use this structure to study the
homology of $\BDiff(M)^{\partial}$ where $M$ is an arbitrary
$(n-1)$-connected, $(2n+1)$-dimensional manifold with boundary
diffeomorphic to $S^{2n}$. In particular if $M$ is now a closed,
$(n-1)$-connected, $(2n+1)$-dimensional manifold, we consider a map
$$s_{g}: \BDiff(M\# W_{g,1})^{\partial} \longrightarrow  \BDiff(M\# W_{g+1, 1})^{\partial}$$
defined in a way similar to the map from (\ref{eq: stabilization map intro}). 
\begin{Conjecture} Let $M$ be a closed, $(n-1)$-connected, $(2n+1)$-dimensional manifold with $n \geq 4$. Then the induced map on homology, 
  $$(s_{g})_{*}: H_{k}(\BDiff(M\# W_{g,1})^{\partial}; \Z)
  \longrightarrow H_{k}(\BDiff(M\# W_{g+1, 1})^{\partial}; \Z)$$ is an
  isomorphism provided that $k \leq \frac{1}{2}(g - 3)$.
\end{Conjecture}
\subsection{Acknowledgements} The author would like to thank Boris Botvinnik for suggesting this particular problem and for numerous helpful discussions on the subject of this paper. 
\newcommand{\M}[1]{\ensuremath{\mathcal{M}_{#1}}}
\section{Basic Definitions} \label{Basic Definitions}
 \subsection{The Moduli Space} \label{The Moduli Spaces} 
Recall from the introduction the manifold 
 $$W_{g, 1} := \#^{g}(S^{n+1}\times S^{n}) - \Int(D^{2n+1}).$$ 
We construct a model for the classifying space $\BDiff(W_{g,1})^{\partial}$. 
Fix a collar embedding,
 $$h: [0, \infty)\times S^{2n} \longrightarrow W_{g,1}$$
 which maps $\{0\}\times S^{2n}$ to $\partial W_{g,1}$. We then
fix a smooth embedding $\phi: S^{2n} \longrightarrow \R^{\infty}$ and denote by $S$ the image of $\phi$ in $\R^{\infty}$.
\begin{defn} Let $\mathcal{M}_{g}$ denote the set of compact $(2n+1)$-dimensional submanifolds 
 $$W \subset [0, \infty)\times \R^{\infty}$$ such that $\partial W =
    \{0\}\times S$, and $[0, \epsilon)\times S \subset W$ for some
      $\epsilon > 0$, and such that $W$ is diffeomorphic relative to
      its boundary to the manifold $W_{g,1}$.

The space $\mathcal{M}_{g}$ is topologized as a quotient of the space of smooth embeddings 
$$\varphi: W_{g,1} \longrightarrow [0, \infty)\times \R^{\infty}$$
for which there exists an $\epsilon > 0$ such that
$$\varphi(h(t, x)) \; = \; (t, \phi(x)) \quad \text{for all $(t, x) \in [0, \epsilon)\times S^{2n}$.}$$ 
In this quotient, two embeddings are identified if they have the same image. 
\end{defn}
We denote by $\Emb(W_{g,1}, [0,1]\times\R^{\infty})^{\partial}$ the
space of embeddings $W_{g,1} \hookrightarrow [0, \infty)\times
  \R^{\infty}$, with the boundary behavior described in the previous
  definition, equipped with the $C^{\infty}$-topology. The group
  $\Diff(W_{g,1})^{\partial}$ of diffeomorphisms of $W_{g,1}$ which
  fix a neighborhood of the boundary pointwise, acts freely (and smoothly) on the space
  $\Emb(W_{g,1}, [0,1]\times\R^{\infty})^{\partial}$ by precomposing
  embeddings by diffeomorphisms.  It is easy to see that the orbit
  space of this group action is homeomorphic to $\mathcal{M}_{g}$.
  The quotient map
$$\Emb(W_{g,1}, [0,1]\times\R^{\infty})^{\partial} \longrightarrow \;  \frac{\Emb(W_{g,1}, [0,1]\times\R^{\infty})^{\partial}}{\Diff(W_{g,1})^{\partial}} \;  \cong \; \mathcal{M}_{g}$$
is a locally trivial fibre-bundle (see \cite{BF 81}), and since the space $\Emb(W_{g,1}, [0,1]\times\R^{\infty})^{\partial}$ is weakly contractible, it follows that there is a weak homotopy equivalence,
\begin{equation} \label{eq: weak equivalence moduli} \BDiff(W_{g,1})^{\partial} \; \sim \; \mathcal{M}_{g}. \end{equation}
The space $\mathcal{M}_{g}$ is sometimes referred to as the \textit{moduli space of manifolds diffeomorphic to $W_{g,1}$.}

Denote by $W_{1, 2}$ the manifold obtained by deleting the interior of a $(2n+1)$-dimensional disk from the interior of $W_{1,1}$, and then denote  by $S^{2n}_{0}$ and $S^{2n}_{1}$, the two boundary components of $W_{1, 2}$, both of which are diffeomorphic to the sphere $S^{2n}$. 
We pick once and for all a collared embedding 
$\theta: W_{1, 2} \hookrightarrow [0, 1]\times \R^{\infty}$, such that 
$$\theta(S_{0}^{2n}) = \{0\}\times S \quad  \text{and}  \quad \theta(S_{1}^{2n}) = \{1\}\times S,$$
where $S \subset \R^{\infty}$ is the submanifold that was used in the definition of $\mathcal{M}_{g}$. 
Using this embedding, we define a map 
$\mathcal{M}_{g} \longrightarrow \mathcal{M}_{g+1}$
by the formula,
$$W \; \mapsto \; \theta(W_{1, 2}) \cup (W + e_{1})  \quad \text{for $W \in \mathcal{M}_{g}$}$$ 
where $W + e_{1}$ denotes translation of the submanifold $W$ by one unit in the first coordinate. Since the ambient space is infinite dimensional, all such embeddings of $W_{1, 2}$ into $[0, 1]\times \R^{\infty}$ are isotopic, thus we have a well defined homotopy class of maps
\begin{equation} \label{eq: stab map}  s_{g}: \mathcal{M}_{g} \longrightarrow \mathcal{M}_{g+1}. \end{equation}
Throughout, we will refer to this map as the \textit{stabilization map}. Using the weak homotopy equivalence from (\ref{eq: weak equivalence moduli}), this induces the map
$$H_{*}(\BDiff(W_{g,1})^{\partial}; \Z) \longrightarrow
  H_{*}(\BDiff(W_{g+1,1})^{\partial}; \Z),$$
  from the statement of our main result, Theorem \ref{thm: Main Theorem}.

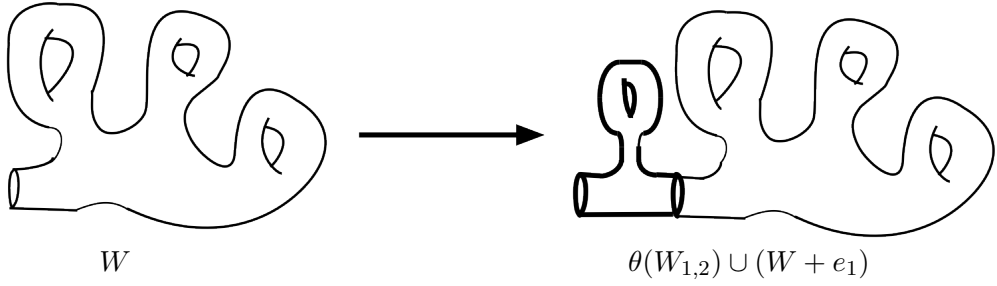
\begin{figure} \label{fig: Stabilization}
\begin{picture}(0, 0)
\put(50, 10){{\small $W$}}
\put(250, 10){{\small $\theta(W_{1, 2})\cup (W + e_{1})$}}
\end{picture}
\includegraphics[scale=.25]{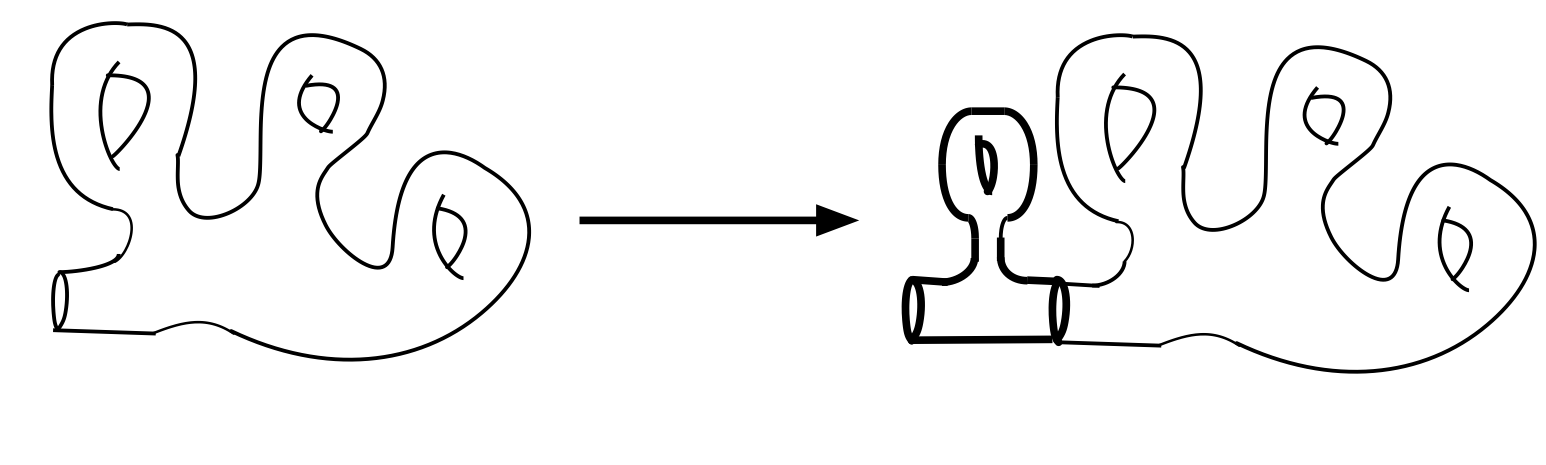}
\caption{The above figure depicts the stabilization map. The submanifold $W \subset [0, \infty)\times \R^{\infty}$ is diffeomorphic to $W_{3,1}$. The stabilization map sends $W$ to the submanifold $\theta(W_{1, 2})\cup (W + e_{1}) \; \subset \; [0, \infty)\times \R^{\infty}$ which is diffeomorphic to $W_{4, 1}$.}
\end{figure}

\subsection{The Complex of Embedded Handles} 
\label{The Complex of Embedded Handles}
We now construct a semi-simplicial space similar to the one
constructed in \cite[Definition 4.1]{GRW 12}.  We will need to
construct a certain manifold as follows. Fix an oriented embedding
$$\beta: \{1\}\times D^{2n} \longrightarrow \partial W_{1,1}.$$
We define $H$ to be the manifold obtained from $W_{1,1}$ by attaching $[0, 1]\times D^{2n}$ to $W_{1,1}$ along the embedding $\beta$, i.e. $H:= W_{1,1}\cup_{\beta}D^{2n}\times[0,1]$. 

We construct a \textit{core} $C
\subset H$ as follows. Recall that $W_{1,1}$ is defined to be the manifold with boundary obtained from $S^{n}\times S^{n+1}$ by deleting the interior of an embedded open disk. Choose base points $a_{0} \in S^{n}$ and
$b_{0} \in S^{n+1}$ such that,
\begin{enumerate}
\item[(a)] the subspace $(S^{n}\times \{a_{0}\})\cup (\{b_{0}\}\times S^{n+1}) \subset S^{n}\times S^{n+1}$ is contained in $W_{1,1} \subset S^{n}\times S^{n+1}$,
\item[(b)] the pair $(a_{0}, -b_{0}) \in S^{n}\times S^{n+1}$ is contained in $W_{1,1}$. 
\end{enumerate}
We then choose an embedded path $\gamma$ in $H = W_{1,
  1}\cup_{\beta}D^{2n}\times[0,1]$ from the point
$(a_{0}, -b_{0})$ to $(0,0) \in [0,1]\times D^{2n}$ whose interior
does not intersect 
$$
(S^{n}\times \{a_{0}\})\cup (\{b_{0}\}\times
S^{n+1})
$$ 
and whose image agrees with $[0,1]\times \{0\} 
\subset [0,1]\times D^{2n}$,
inside $[0,1]\times D^{2n}$, see Fig. \ref{Fig2}.

We then define $C$ to be the subspace of $H$
given by
$$
C:=(S^{n}\times \{a_{0}\})\cup (\{b_{0}\}\times S^{n+1})\cup \gamma([0,1])\cup (\{0\}\times D^{2n})
\; \subset \; H.
$$
It is immediate that $C$ is homotopy equivalent to to $S^{n}\vee S^{n+1}$ and that furthermore $C$ is a deformation retract of $H$. 
\vspace{5mm}

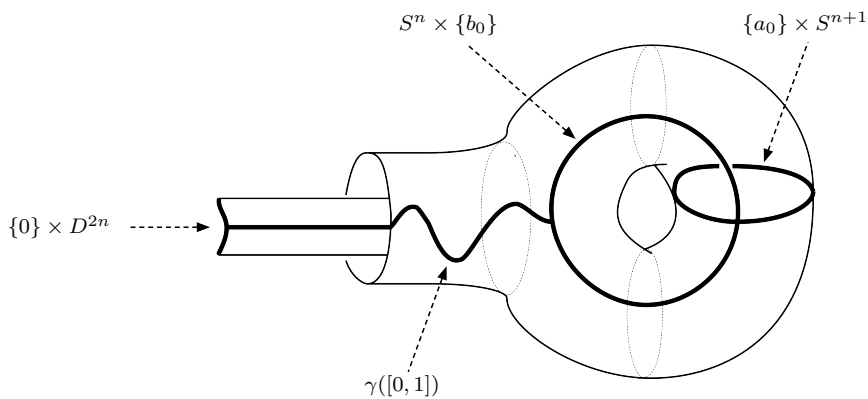
\begin{figure}[!htb]
\label{fig: The Core 2}
\begin{picture}(0,0)
\put(-40, 55){{\tiny $\{0\}\times D^{2n}$}}
\put(238, 132){{\tiny $\{a_{0}\}\times S^{n+1}$}}
\put(108, 132){{\tiny $S^{n}\times\{b_{0}\}$}}
\put(95, -5){{\tiny $\gamma([0,1])$}}
\end{picture}
\includegraphics[height=1.8in]{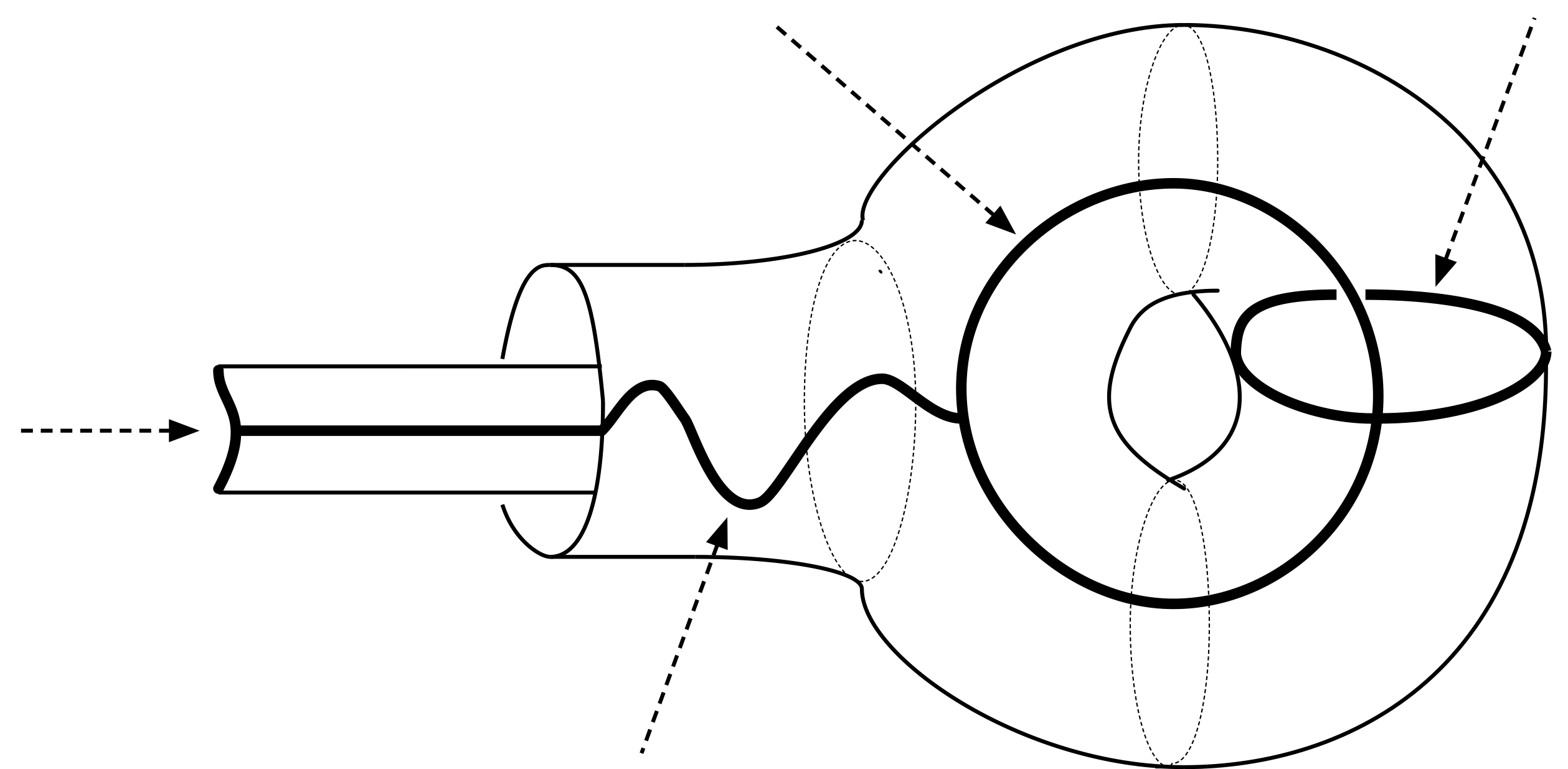}
\caption{Pictured above is the manifold $H = [0, 1]\times D^{2n}
  \cup_{\beta} W_{1,1}$. There is an embedding of the core, $C =
  (S^{n}\times\{b_{0}\})\cup(\{a_{0}\}\times S^{n+1})\cup \gamma([0,1])\cup (\{0\}\times
  D^{2n})$, which is a deformation retraction of $H$.}\label{Fig2}
\end{figure}

We choose an isotopy of
embeddings
$$\rho_{t}: H \longrightarrow H$$
defined for $t \in [0,\infty)$ which satisfies the following conditions:
\begin{enumerate} 
\item[i.] $\rho_{0} = Id_{H}$, 
\item[ii.] for any open neighborhood $U$ containing $C$, there exists $t_{U} \in [0,\infty)$ such that\\
 $\rho_{t}(H) \subset U$ for all $t > t_{U}$,
\item[iii.] for each $t \in  [0,\infty)$ there is an open neighborhood  $V_{t}$ of $C$ such that $\rho_{t}\vert_{V_{t}} \; = \; Id_{V_{t}}$. 
\end{enumerate}

Let $W$ be a compact $(2n+1)$-dimensional manifold with $\partial W\neq \emptyset$, equipped with an embedding 
of a base coordinate patch 
$$a: [0,1)\times\R^{2n} \longrightarrow W$$ such that $a^{-1}(\partial
  W) = \{0\}\times\R^{2n}$.  We define two semi-simplicial spaces
  $K_{\bullet}(W, a)$ and $K_{\bullet}^{\delta}(W, a)$ as follows.
\begin{defn} \label{defn: the embedding complex} Let $W$ and $a: [0,1)\times\R^{2n} \longrightarrow W$ be as above. The spaces of $p$-simplices, $K_{p}(W, a)$ and $K^{\delta}_{p}(W, a)$ are defined as follows:
\begin{enumerate} 
\item[(i)] 
Let $K_{0}(W, a)$ be the space of pairs $(t, \phi)$, where $t \in \R$ and 
$\phi: H \longrightarrow W$ is an embedding for which there exists $\epsilon > 0$ such that for 
$(s, z) \in [0, \epsilon)\times D^{2n} \; \subset \; H$, \; 
$\phi(s, z) = a(s, z + te_{1})$
where, $e_{1} \in \R^{2n}$ denotes the first basis vector. 
\item[(ii)] Let $K_{p}(W, a) \subset (K_{0}(W, a))^{p+1}$ consist of those tuples $((t_{0}, \phi_{0}), \dots, (t_{p}, \phi_{p}))$ such that $t_{0} < \cdots < t_{p}$ and the intersections  $\phi_{i}(C) \cap \phi_{j}(C)$ are empty for all $i , j \in \{0, \dots, p\}$. 
\end{enumerate}

The spaces $K_{p}(W, a)\subset (\R\times
  \Emb(H,W))^p$ are topologized using the $C^{\infty}$-topology on
the space of embeddings. 
Then the space $K^{\delta}_{p}(W, a)$ is defined so as to have the same
underlying set as $K_{p}(W, a)$, but topologized using the discrete
topology.  

The assignments $[p] \mapsto K_{p}(W, a)$ and $[p] \mapsto
K_{p}^{\delta}(W, a)$ define semi-simplicial spaces denoted by 
$K_{\bullet}(W, a)$ and
$K_{\bullet}^{\delta}(W, a)$ respectively.
Here the face maps
$d_{i}: K_{p}(W, a) \longrightarrow K_{p-1}(W, a)$ for \; $0 \leq i
\leq p$ are defined by sending a $p$-tuple $((t_{0}, \phi_{0}), \dots,
(t_{p}, \phi_{p}))$ to the $(p-1)$-tuple obtained by deleting the
$i$th entry.  The face maps $d_{i}: K^{\delta}_{p}(W, a)
\longrightarrow K^{\delta}_{p-1}(W, a)$ are defined similarly. 
\end{defn} 
We now define a simplicial complex $K^{c}(W, a)$
  related to the semi-siplicial space $K^{\delta}_{\bullet}(W, a)$.
\begin{defn} 
Let $K^{c}(W, a)$ be the simplicial complex with the set of vertices  identified with $K^{\delta}_{0}(W, a)$.  A set of vertices $\{(t_{0}, \phi_{0}), \dots, (t_{p}, \phi_{p})\}$, is defined to be a $p$-simplex in $K^{c}(W, a)$ if when written with $t_{0} \leq \cdots \leq t_{p}$, the list $((t_{0}, \phi_{0}), \dots, (t_{p}, \phi_{p}))$ is an element of $K^{\delta}_{p}(W, a)$. 
\end{defn}
Since a $p$-simplex of $K^{\delta}_{\bullet}(W, a)$ is
  determined by its unordered set of vertices, there is a natural
  homeomorphism of geometric realizations: 
\begin{equation}\label{eq-new}
|K^{\delta}_{\bullet}(W, a)|
  \cong |K^{c}(W, a)|.
\end{equation}

\begin{remark}
The semi-simplicial spaces $K_{\bullet}(W, a)$ and
  $K_{\bullet}^{\delta}(W, a)$ depend on the choice of embedding $a:
          [0,1)\times\R^{2n} \longrightarrow W$.
         However, it is easy to see that the
            homeomorphism type of their geometric realizations does
            not depend on the choice of $a$. For the
            duration of the paper we will suppress $a$ from the
            notation and denote 
$$
K_{\bullet}(W) := K_{\bullet}(W, a), \ \ \ 
            K_{\bullet}^{\delta}(W) :=K_{\bullet}^{\delta}(W, a),  \ \ \ \mbox{and}
\ \ \ K^{c}(W):=K^{c}(W, a).
$$
\end{remark}

\section{The Algebraic Invariants} \label{The Algebraic Invariants}
\subsection{Intersection Products and Vector Bundles} \label{Intersection Products and Vector Bundles}   
In this section we develop in detail the algebraic invariants
associated to an $(n-1)$-connected, $(2n+1)$-dimensional manifold that
were introduced in the introduction. We first state Haefliger's theorem from
\cite{Ha 62} which will allow us to identify the homotopy groups
$\pi_{n+1}(W_{g,1})$ with the set of isotopy classes of embeddings of
$S^{n+1}$ into $W_{g,1}$.
\begin{theorem} {\rm (Haefliger, \cite{Ha 62})}
\label{thm: Haefliger embedding} Let $V$ and $M$ be smooth manifolds with $\dim(V) = l$ and $\dim(M) = m$. Suppose further that $V$ is $(k-1)$-connected and that $M$ is $k$-connected for some integer $k \geq 0$. Then,
\begin{enumerate}
\item[(a)] Any continuous map of $V$ to $M$ is homotopic to a smooth
  embedding provided that $m \geq 2l - k +1$.
\item[(b)] Any two smooth embeddings of $V$ into $M$ which are homotopic as continuous maps are smoothly isotopic provided that $m \geq 2l - k + 2$. 
\end{enumerate}
\end{theorem}
Setting $V = S^{n+1}$ and $M = W_{g,1}$ yields the following corollary:
\begin{corollary} Let $n \geq 4$. Then $\pi_{n+1}(W_{g,1})$ and $\pi_{n}(W_{g,1})$ are each in bijective correspondence with set of isotopy classes of embeddings of $S^{n+1}$ and $S^{n}$ into $W_{g,1}$ respectively. \end{corollary} 
Now, let $M$ be a $k$-connected smooth manifold of dimension $m$. Let $r$ and $t$ be positive integers less than or equal to the number 
$\min\{\frac{1}{2}(m + k - 2), \; m-2\}.$ 
We study 
maps, 
\begin{equation}
\lambda_{r,t}: \pi_{r}(M)\times \pi_{t}(M) \; \longrightarrow  \; \pi_{r}(S^{m-t}) \quad \text{and} \quad \alpha_{r}: \pi_{r}(M) \; \longrightarrow  \; \pi_{r-1}(SO_{m-r}) 
\end{equation}
which were
defined by C.T.C. Wall in \cite{Wa 63} as follows. 
For an embedding
$g: S^{r} \longrightarrow M$
(representing an element of $\pi_{r}(M)$),
the class $\alpha_{r}([g]) \in \pi_{r-1}(SO_{m-r})$ is defined to be the element which classifies the normal bundle of $g$. This is well defined because of Theorem \ref{thm: Haefliger embedding}; all homotopies of $g$ can be assumed to be isotopies of smooth embeddings.

Defining $\lambda_{r, t}$ will take more work. Let 
$$g: S^{r} \longrightarrow M \quad \text{and} \quad f: S^{t} \longrightarrow M$$
be embeddings. Let $N_{f} \subset M$ denote a closed tubular neighborhood of $f(S^{t})$ and let $\pi:  N_{f} \longrightarrow f(S^{t})$ denote the bundle projection. Let $U \subset S^{t}$ denote a closed neighborhood diffeomorphic to a disk. We set 
$$V_{0} := \pi^{-1}(f(U)) \quad \text{and} \quad V_{1} := \pi^{-1}[f(S^{t}) - \Int(f(U))].$$ 
We define
$$T_{f}: [M - \Int(V_{0})] \; \longrightarrow \; S^{m-t}$$
to be the \textit{Pontryagin-Thom map} associated to the normal bundle $\pi\vert_{V_{1}}: V_{1} \longrightarrow (S^{t}- \Int(U)$). More precisely, letting $\varphi: V_{1} \stackrel{\cong} \longrightarrow D^{t}\times D^{m-t}$ be a trivialization, and identifying $S^{m-t}$ with the push-out $D^{m-t}\cup_{c} \{\infty\}$ (with $c : S^{m-t-1} \rightarrow \{\infty\}$ the constant map),
$T_{f}$ is defined by the formula 
\begin{equation} \label{eq: pontryagin thom} x \mapsto \begin{cases} 
\text{proj}_{D^{m-t}}(\phi(x)) & \text{if $x \in V_{1}$}\\
\infty & \text{if $x \in \; [M - \Int(V_{0})] - V_{1}$.}
\end{cases}
\end{equation}
Clearly if $f$ and $f'$ are isotopic embeddings, then $T_{f}$ and $T_{f^{'}}$ will be homotopic (the class of $T_{f}$ does not depend on the choice of trivialization $\varphi$ since the base-space is contractible).
Since $r \leq m-2$ and $V_{0} \cong D^{m}$, the inclusion map $i: M - \Int(V_{0}) \hookrightarrow M$ induces an isomorphism 
$$i_{*}: \pi_{r}(M - \Int(V_{0})) \stackrel{\cong} \longrightarrow  \pi_{r}(M).$$
Finally, we define
\begin{equation} \lambda_{r, t}([g], [f]) \; = \; (T_{f})_{*}\circ i_{*}^{-1}[g]. \end{equation}
There is one more important map defined by C.T.C Wall in \cite{Wa 63}
which we need to use in order to compute the value of $\alpha_{r}$ on
certain isotopy classes. Let $2m \geq 3r + 3$, $m \geq 2r - t + 3$,
and $r \geq t$. Then there is a map,
\begin{equation} \label{eq: bundle composition} F: \pi_{t-1}(SO_{m-t})\times \pi_{r}(S^{t}) \longrightarrow \pi_{r-1}(SO_{m-r}) \end{equation}
defined as follows. For $(x, y) \in \pi_{t-1}(SO_{m-t})\times \pi_{r}(S^{t})$, let $E \longrightarrow S^{t}$ be the disk-bundle with fibre-dimension $m - t$, classified by the element $x$. Denote by $i_{0}: S^{t} \rightarrow E$ the zero section. If $f: S^{r} \longrightarrow S^{t}$ is a map representing $y$, the dimensional assumptions on $t, r,$ and $m$ imply that the composition
$$S^{r} \stackrel{f} \longrightarrow S^{t}  \stackrel{i_{0}} \longrightarrow E$$
may be deformed to an embedding, which we denote by $\widehat{f}: S^{r} \longrightarrow E$. We then define $F(x, y) \in \pi_{r-1}(SO_{m-r})$ to be the element which classifies the normal bundle of $\widehat{f}(S^{r})$ in $E$. 

We will need an important result from \cite{Wa 63} regarding this map $F$. 
\begin{lemma} {\rm (Wall, \cite{Wa 63})}
\label{lemma: linear bundle comp} Let $y \in \pi_{r}(S^{t})$ be in the image of the suspension map
$S: \pi_{r-1}(S^{t-1}) \longrightarrow \pi_{r}(S^{t})$. Then the map 
$\pi_{t-1}(SO_{m-t}) \longrightarrow \pi_{r-1}(SO_{m-r})$ given by  $x \; \mapsto \; F(x, y)$
is a homomorphism. 
\end{lemma}

The map $F$ can be used to compute the value of $\alpha_{r}$ on homotopy classes of $\pi_{r}(M)$ which can be factored as a composition 
$S^{r} \stackrel{f} \longrightarrow S^{t} \stackrel{g} \longrightarrow M.$
In this case, it follows directly from the definition of $F$ that
\begin{equation} \label{eq: bundle composition formula}
\alpha_{r}( [ g\circ f ] ) \; = \; F(\alpha_{t}( [ g ] ), \; [ f ]). 
\end{equation}
This formula will prove useful in the next section when computing the value of $\alpha_{n+1}$ on $\pi_{n+1}(W_{g,1})$.

We now specialize to our main case of interest where $M$ is an $(n-1)$-connected, $(2n+1)$-manifold with $n \geq 4$. The map $\lambda_{n, n+1}$ was defined using the \text{Pontryagin-Thom} construction and thus can be interpreted as an intersection invariant. It is easy to deduce:
\begin{proposition} \label{prop: hurewicz intersection}  For elements  $x \in \pi_{n}(M)$ and $y \in \pi_{n+1}(M)$, 
 let $k_{x, y} \in \Z$  be the signed intersection number for transversal embeddings which represent elements $x$ and $y$.
We have, $\lambda_{n, n+1}(x, \; y) \; = \; k_{x, y}\cdot \iota_{n}$
where $\iota_{n}$ is the standard generator of the group $\pi_{n}(S^{n})$. 
\end{proposition}
Let $x$ and $y$ be as in the previous proposition and let $f: S^{n}
\rightarrow M$ and $g: S^{n+1} \rightarrow M$ be transversal
embeddings which represent $x$ and $y$.  Then by the \textit{Whitney
  Trick}, \cite[Theorem 6.6]{Mil 65}, one can find an isotopy $h_{s}:
S^{n+1} \longrightarrow M$, defined for $s \in [0, 1]$ with $h_{0} =
g$ and such that the image of $h_{1}$ intersects the image of $f$
transversally at exactly $|k_{x, y}|$-many points.

Now, let $f: S^{n+1} \rightarrow M$ and $g: S^{n+1} \rightarrow M$ be
embeddings which intersect transversally.  We will need a generalized
version of the Whitney Trick that will enable us to construct
an isotopy $h_{t}: S^{n+1} \rightarrow M$ such that $h_{0} = g$ and
$h_{1}(S^{n+1})\cap f(S^{n+1}) = \emptyset$ if and only if
$\lambda_{n+1, n+1}([f], [g]) = 0$. In \cite{We 67} and \cite{HQ 74},
such generalized versions of the Whitney Trick were
developed. We have:
\begin{proposition} {\rm (Wells, \cite{We 67})} \label{prop: Generalized Whitney Trick} 
Let $M$ be an $(n-1)$-connected, $(2n+1)$-dimensional manifold. Let
$f: S^{n+1} \longrightarrow M$ and $g: S^{n+1} \longrightarrow M$ be
embeddings such that $\lambda_{n+1, n+1}([f], [g]) = 0$. Then there is
exists an isotopy $h_{t}: S^{n+1} \rightarrow M$ such that $h_{0} = g$
\text{and} $h_{1}(S^{n+1})\cap f(S^{n+1}) = \emptyset.$ Furthermore,
if $f$ and $g$ are transversal and $U \subset S^{n+1}$ is an open disk
containing $g^{-1}(f(S^{n+1})\cap g(S^{n+1}))$, then the isotopy
$h_{t}$ may be chosen so that $h_{t}(x) = g(x)$ for all $x \in
S^{n+1}\setminus U$ and $t \in [0, 1]$.
\end{proposition}
\begin{remark} This is a specialization of the main theorem of \cite{We 67}. In \cite{We 67} the author associates to sub-manifolds $A^{r}, B^{t} \subset X^{m}$ with $\partial A, \partial B \subset \partial X$ and $(\partial A)\cap(\partial B) = \emptyset$, an invariant valued in the stable homotopy group
$\alpha(A, B; X) \in \pi^{S}_{r+t-m}$
and proves that if certain dimensional and connectivity conditions are
met, then there exists an isotopy $h_{t}: A^{r} \rightarrow X$ with
$h_{0} = i_{A}$ and $h_{1}(A)\cap B = \emptyset$ if and only if
$\alpha(A, B; X) = 0$.  
Furthermore, it is proven that this isotopy
can be chosen so as to be constant on $\partial A$.  It can be easily
checked that by setting $A := g(S^{n+1})$, $B := f(S^{n+1})$, and $X
:= M$, the invariant $\alpha(A, B; X)$ is equal to $\lambda_{n+1,
  n+1}([f], [g])$ after identifying $\pi_{n+1}(S^{n})$ (for $n \geq
4)$ with $\pi_{1}^{S}$. 

We emphasize that with $n \geq 4$, the dimensional and connectivity
conditions from the main theorem of \cite{We 67} are satisfied. In
order to construct such an isotopy $h_{t}$ that is constant outside of
a neighborhood $U \subset S^{n+1}$ of $g^{-1}(f(S^{n+1})\cap
g(S^{n+1}))$, we set $A = g(S^{n+1} - U)$, $B = f(S^{n+1})$, and $X =
M - N(U)$ where $N(U) \subset M$ is a tubular neighborhood of $U$ in
$M$. In this case $\alpha(A, B; X)$ is still equal to $\lambda_{n+1,
  n+1}([f], [g])$, since $\lambda_{n+1, n+1}([f], [g])$ was defined
precisely by deleting such neighborhoods from $S^{n+1}$ and $M$ and
then applying the Pontryagin-Thom construction (\ref{eq: pontryagin
  thom}).
\end{remark} 
We now state an important theorem regarding $\lambda_{n+1,n+1}$ and
$\alpha_{n+1}$ which was proven in \cite{Wa 63}. For what follows,
denote by $\pi: SO_{n+1} \longrightarrow SO_{n+1}/SO_{n} \cong S^{n}$
the projection map, denote by $\partial: \pi_{n+1}(S^{n})
\longrightarrow \pi_{n}(SO_{n})$ the boundary map in the long-exact
sequence associated to the fibre-sequence $SO_{n} \rightarrow SO_{n+1}
\rightarrow S^{n}$, and finally denote by $S: \pi_{n}(S^{n-1})
\rightarrow \pi_{n+1}(S^{n})$ the suspension map.

\begin{theorem} {\rm (Wall, \cite{Wa 63})} \label{thm: Wall intersection formulas} 
Let $M$ be an $(n-1)$-connected, $(2n+1)$-dimensional manifold with $n \geq 4$. Then for $x, y \in \pi_{n+1}(M)$, we have
$$\lambda_{n+1, n+1}(x, x) = S\pi\alpha_{n+1}(x),$$
$$\alpha_{n+1}(x + y) \; = \; \alpha_{n+1}(x) + \alpha_{n+1}(y) + \partial (\lambda_{n+1, n+1}(x, y)).$$
Furthermore, $\lambda_{n+1, n+1}$ is $(-1)^{(n+1)}$-symmetric bilinear.
\end{theorem}

Here is one more useful result regarding $\lambda_{n, n+1}$ and $\lambda_{n+1, n+1}$ from \cite{Wa  63}. 

\begin{proposition}  {\rm (Wall, \cite{Wa 63})} \label{prop: lambda composition} Let $M$ be an $(n-1)$-connected, $(2n+1)$-dimensional manifold with $n \geq 4$. Let $f: S^{n} \rightarrow M$ and $g: S^{n+1} \rightarrow M$ be embeddings and let $\beta: S^{n+1} \longrightarrow S^{n}$ be a map. Then $\lambda_{n+1, n+1}([f\circ \beta], \; [g]) \; = \; \lambda_{n, n+1}([f], \; [g]) \circ [\beta].$ \end{proposition}

We will use these maps to deduce important information about the structure of $H$ and $W_{g,1}$. 

\subsection{Calculations} \label{The Homology} 
We now compute the values of $\lambda_{n, n+1}$, $\lambda_{n+1, n+1}$, $\alpha_{n}$, and $\alpha_{n+1}$ on the manifolds $W_{g,1}$. 
Recall from the previous section the space, 
$$H := W_{1,1}\cup_{\beta}[0,1]\times D^{2n}$$ 
and the core $C \subset H$ used in Definition \ref{defn: the embedding complex}.  

Denote by $\iota_{n}$ and $\iota_{n+1}$ the standard generators of $\pi_{n}(S^{n})$ and $\pi_{n+1}(S^{n+1})$ respectively. Recall that 
$$\pi_{n+1}(S^{n}) \cong  \begin{cases}
\Z & \text{if $n = 2$}\\
\Z/2 & \text{if $n \geq 3$.}
\end{cases}$$
We then set 
$\rho_{n+1} \in \pi_{n+1}(S^{n})$ 
to be the generator represented by the $(n-2)$-fold suspension of the \text{Hopf-Fibration} $S^{3} \rightarrow S^{2}$. 
Letting,
\begin{equation}  \label{representative embeddings}
j_{n}: S^{n} \hookrightarrow C,  \quad  j_{n+1}: S^{n+1} \hookrightarrow C, \quad  \text{and} \quad 
i_{C}: C \hookrightarrow H
\end{equation}
denote the inclusions, we set
\begin{equation}
\begin{aligned}
e := (i_{C}\circ j_{n})_{*}(\iota_{n}) \; \in \pi_{n}(H), \quad \eta := (i_{C}\circ j_{n+1})_{*}(\iota_{n+1}) \; \in \pi_{n+1}(H),\\
\end{aligned}
\end{equation}
and then
\begin{equation} \mu \; := e \; \circ \rho_{n+1} \in \pi_{n+1}(H) \end{equation}
where, with some abuse of notation, $e \circ \rho_{n+1}$ is the homotopy class of the composition of maps representing $e$ and $\rho_{n+1}$. It is easy to see that
$$\pi_{n}(H) = \langle e \rangle \cong \Z  \quad  \text{and} \quad   \pi_{n+1}(H) = \langle \eta, \;  \mu \rangle \cong \Z\oplus \Z/2.$$

We are now in a position to compute the values of $\lambda_{n, n+1}$, $\lambda_{n+1, n+1}$, $\alpha_{n}$, and $\alpha_{n+1}$ on the homotopy groups of $H$. 

\begin{proposition} \label{prop: intersection pairing 1} On the homotopy groups of $H$, $\lambda_{n, n+1}$, $\lambda_{n+1, n+1}$, $\alpha_{n}$, and $\alpha_{n+1}$ take the following values: \\
\begin{equation}
\xymatrix@R-1pc{
 \alpha_{n}(e) = 0 &  \alpha_{n+1}(\eta) = 0 & \alpha_{n+1}(\mu) = 0\\
\lambda_{n, n}(e, e) = 0  & \lambda_{n+1, n+1}(\eta, \eta) = 0  & \lambda_{n+1, n+1}(\mu, \mu) = 0\\
 \lambda_{n, n+1}(e, \eta) = \iota_{n} &  \lambda_{n+1, n+1}(\eta, \mu) = \rho_{n+1} & \lambda_{n, n+1}(e, \mu) = 0. }
\end{equation}
\end{proposition}
\begin{proof} The calculations $\alpha_{n}(e) = 0$ and $\alpha_{n+1}(\eta) = 0$ follow from the fact that the embeddings $i_{C}\circ j_{e}$ and $i_{C}\circ j_{\eta}$ from (\ref{representative embeddings}) which represent the classes $e$ and $\eta$ have trivial normal bundles. 
The calculation $\lambda_{n, n+1}(e, \eta) = \iota_{n}$ follow from the fact that the embeddings $i_{C}\circ j_{e}$ and $i_{C}\circ j_{\eta}$ intersect transversally at exactly one point. 

To see that $\alpha_{n+1}(\mu) = 0$ we recall that $\mu$ is defined by a composition, namely $\mu = e \circ \rho_{n+1}$. From (\ref{eq: bundle composition formula}) we see that $\alpha_{n+1}(e \circ \rho_{n+1}) = F(\; \alpha_{n}(e) \;, \; \rho_{n+1}) = F(0, \; \rho_{n+1})$. Since $n$ is assumed to be greater than or equal to $4$, the element $\rho_{n+1}$ is a suspension, namely the $(n-2)$-fold suspension of the \textit{Hopf fibration}. From this, Lemma \ref{lemma: linear bundle comp} implies that $F(\;  \cdot \;, \; \rho_{n+1})$ is linear in the first variable, hence $F(0, \rho_{n+1}) = 0$, and so $\alpha_{n+1}(e \circ \rho_{n+1}) \; = \; 0$.

The calculation $\lambda_{n,n}(e, e) = 0$ follows trivially from the fact that $\pi_{n}(S^{n+1}) = 0$. The calculation $\lambda_{n+1, n+1}(\eta, \eta) = 0$ follows from the fact that $\alpha_{n+1}(\eta) = 0$, using Theorem \ref{thm: Wall intersection formulas}. The calculation $\lambda_{n+1, n+1}(\mu, \mu) = 0$ follows from $\alpha_{n+1}(\mu) = 0$ using Theorem \ref{thm: Wall intersection formulas} as well. 

Finally, to compute $\lambda_{n+1, n+1}(\eta, \mu)$ we use Proposition \ref{prop: lambda composition} to write,
$$\lambda_{n+1, n+1}(\eta, \mu) \; = \; \lambda_{n+1, n+1}(\eta, e \circ \rho_{n+1}) \; =\;  \lambda_{n+1, n}(\eta, e)\circ \rho_{n+1} \; =\;  \iota_{n}\circ \rho_{n+1} \; = \; \rho_{n+1}.$$ \end{proof}

We now compute the values of these invariants on $\pi_{k}(W_{g,1})$. Since $W_{g, 1}$ is diffeomorphic to the boundary connect sum of $g$-copies of $W_{1, 1}$ (or equivalently $H$), we see that 
$$\pi_{n}(W_{g}) \cong (\Z)^{\oplus g} \quad \text{and} \quad \pi_{n+1}(W_{g,1}) \; \cong \; \Z^{\oplus g}\oplus (\Z/2)^{\oplus g}.$$
Fix a $(g-1)$ simplex, 
\begin{equation} \label{eq: standard reference embeddings} \{(\psi_{0}, t_{0}), \dots, (\psi_{g-1}, t_{g-1})\} \in K^{c}(W_{g,1}). 
\end{equation}
Such a $(g-1)$-simplex exists since the manifold $W_{g,1}$ is by construction diffeomorphic to the boundary connected sum of $g$-many copies of $H$.  For $i = 0, \dots, g-1$, we set
$$e_{i} := (\psi_{i})_{*}(e) \in \pi_{n}(W_{g,1}), \quad \eta_{i} := (\psi_{i})_{*}(\eta) \in \pi_{n+1}(W_{g,1}), \quad \mu_{i} := (\psi_{i})_{*}(\mu) \in \pi_{n+1}(W_{g,1}).$$
%
\begin{proposition} \label{prop: intersection pairings 2} On the homotopy groups of $W_{g,1}$, the maps $\lambda_{r, t}$ and $\alpha_{r}$ take the following values: \\
\begin{equation}
\xymatrix@R-1pc{
\alpha_{n}(e_{i}) = 0 & \alpha_{n+1}(\eta_{i}) = 0 & \alpha_{n+1}(\mu_{i}) = 0\\
\lambda_{n, n+1}(e_{i}, \eta_{j}) = \delta_{i,j}\cdot \iota_{n} &  \lambda_{n, n}(e_{i}, e_{j}) = 0  & \lambda_{n+1, n+1}(\eta_{i}, \eta_{j}) = 0  \\
\lambda_{n+1, n+1}(\eta_{j}, \mu_{i}) = \delta_{i,j}\cdot\rho_{n+1} &  \lambda_{n+1, n+1}(\mu_{i}, \mu_{j}) = 0  & \lambda_{n, n+1}(e_{i}, \mu_{j}) = 0.
}
\end{equation}
\end{proposition}
\begin{proof} This follows directly from the previous proposition combined with the fact that the intersection 
$\psi_{i}(C)\cap\psi_{j}(C) = \emptyset$ for all $i, j \in \{0, \dots g-1\}$. 
\end{proof}

\subsection{Wall Pairings} \label{subsection: Wall Pairings}
We now formalize the algebraic structure given by the $5$-tuple, 
$$(\pi_{n}(H), \; \pi_{n+1}(H), \; \lambda_{n, n+1}, \; \lambda_{n+1, n+1}, \; \alpha_{n+1}).$$
For an integer $g \geq 0$, let $X$ and $Y$ be $\Z$-modules (not necessarily free) of rank $g$. Let 
$$\lambda:  X\times Y \longrightarrow  \pi_{n}(S^{n}), \quad  q: Y\times Y \longrightarrow \pi_{n+1}(S^{n}), \quad \alpha: Y \longrightarrow \pi_{n}(SO_{n})$$
be maps such that $\lambda$ is bilinear, $q$ is symmetric bilinear,  and $\alpha$ is a function (not necessarily  a homomorphism) that satisfies the relation
\begin{equation} \label{eq: alpha summation formula} \alpha(a + b) = \alpha(a) + \alpha(b) + \partial (q(a, b)) \; \; \text{for all $a, b \in Y$}. \end{equation}

The main example of such maps arise by setting $X := \pi_{n}(M)$ and $Y := \pi_{n+1}(M)$ for $M$ an $(n-1)$-connected, $(2n+1)$-dimensional manifold, and setting the maps $\lambda, q,$ and $\alpha$ equal to  $\lambda_{n, n+1}, \lambda_{n+1, n+1},$ and $\alpha_{n+1}$  as defined in Section \ref{Intersection Products and Vector Bundles} (we are ignoring $\alpha_{n}$ here since in our case of interest it is identically zero).
We now impose further conditions on $X,\;  Y, \; \lambda, \; q,$ and $\alpha$.

\begin{defn} \label{defn: wall pairing} Let $X,\;  Y, \; \lambda, \; q,$ and $\alpha$ be as defined above with 
$$X \cong \Z^{\oplus g} \quad \text{and} \quad Y \cong \Z^{\oplus g}\oplus (\Z/2)^{\oplus g}.$$ 
Suppose further that there exist bases
$$(x_{1}, \dots, x_{g}) \quad \text{and} \quad (y_{1}, \dots, y_{g}, z_{1}, \dots, z_{g})$$
of $X$ and $Y$ respectively with $(z_{1}, \dots, z_{g})$ a basis for the $\Z/2$-component of $Y$, such that for all $i, j \in \{1, \dots, g\}$ the following conditions are satisfied:
\begin{enumerate}
 \item[i.]  $\lambda(x_{i}, y_{j}) = \delta_{i, j}\cdot\iota_{n}$ and  $\lambda(x_{i}, z_{j}) = 0$,
 \item[ii.] $q(y_{i}, z_{j}) = q(z_{j}, y_{i}) = \delta_{i, j}\cdot\rho_{n+1}$ and  $q(y_{i}, y_{j}) = q(z_{i}, z_{j}) = 0$, 
\item[iii.] $\alpha(y_{i}) = \alpha(z_{i}) = 0$.
\end{enumerate}
These conditions determine the values of the functions $\lambda, q$, and $\alpha$.
With these conditions satisfied,
the $5$-tuple
$(X, \; Y, \; \lambda, \; q, \; \alpha)$
is said to be a \textit{Wall pairing} of rank $g$. 
\end{defn}
It follows immediately from the calculation of Proposition \ref{prop:
  intersection pairing 1} that
\begin{equation} \label{eq: elementary wall pairing} (\pi_{n}(H), \; \pi_{n+1}(H), \; \lambda_{n, n+1}, \; \lambda_{n+1, n+1}, \; \alpha_{n+1}) \end{equation}
is a Wall pairing of rank $1$.  It follows from Proposition \ref{prop: intersection pairing 1} that the bases $(e)$ and $(\eta, \mu)$, and  maps  $\lambda_{n, n+1}, \lambda_{n+1, n+1}$, and $\alpha_{n+1}$, satisfy of the conditions put forth the Definition \ref{defn: wall pairing}.
It then follows from Proposition \ref{prop: intersection pairings 2} that 
$$(\pi_{n}(W_{g,1}), \; \pi_{n+1}(W_{g,1}), \; \lambda_{n, n+1}, \; \lambda_{n+1, n+1}, \; \alpha_{n+1})$$
is a \text{Wall pairing} of rank $g$.

Homorphisms and direct sums of Wall pairings are defined in the obvious way. 
We state without proof the easy to verify but important proposition:
\begin{proposition} \label{prop: rank 1 wall} Any Wall pairing of rank one is isomorphic to the Wall pairing
(\ref{eq: elementary wall pairing}).  Furthermore, any \text{Wall
    pairing} of rank $g$ is isomorphic to the $g$-fold direct sum of
  the Wall pairings (\ref{eq: elementary wall
    pairing}). \end{proposition} By this proposition we are justified
in thinking of a rank $g$ Wall pairing as an algebraic version of the
manifold $W_{g,1}$.

Let $(X, \; Y, \; \lambda, \; q, \; \alpha)$ be a Wall pairing of rank $g$. It follows from Definition \ref{defn: wall pairing} that 
$\lambda: X\times Y \longrightarrow \pi_{n}(S^{n})$ is a \textit{perfect pairing} , i.e. the map 
$$X \longrightarrow \Hom(Y, \pi_{n}(S^{n})), \quad x \mapsto
\lambda(x, \; \cdot \;)$$ is an isomorphism. The bilinear pairing $q:
Y\times Y \longrightarrow \pi_{n+1}(S^{n})$ enjoys a similar property.
\begin{proposition} Let $Y' \subset Y$ be 
any compliment to $\Torsion(Y)$ in $Y$. Then the map 
$$\Torsion(Y) \longrightarrow \Hom(Y', \pi_{n+1}(S^{n})), \quad z
\mapsto q(z, \; \cdot \;)$$ is an isomorphism. \end{proposition}
\begin{proof} Let $(y_{1}, \dots, y_{g}, z_{1}, \dots, z_{g})$ be a basis of $Y$ which satisfies conditions  i., ii., and iii. in Definition \ref{defn: wall pairing} and denote by $Y''$ the free sub $\Z$-module generated by $(y_{1}, \dots, y_{g})$. Since $\text{Torsion}(Y)$ is generated by $(z_{1}, \dots, z_{g})$, 
it is immediate from Definition \ref{defn: wall pairing} that the map 
$$\text{Torsion}(Y) \longrightarrow \Hom(Y'', \pi_{n+1}(S^{n})), \quad z \mapsto q(z, \; \cdot \;)$$
is an isomorphism. 
Let $\pi: Y \longrightarrow Y'$ be the projection map defined by the direct-sum decomposition, $Y'\oplus\text{Torsion}(Y) = Y$. For dimensional reasons the restriction map, $\pi|_{Y''} : Y'' \longrightarrow Y'$ is an isomorphism.  Since $q(z, z') = 0$ whenever both $z, z'$ are in $\text{Torsion}(Y)$ it follows that 
$$q(z, y) \; = \; q(z, \pi(y)) \quad \text{for all $z \in \text{Torsion}(Y)$ and $y \in Y''$}.$$
This implies that the diagram
$$\xymatrix{
\Torsion (Y) \ar[rr] \ar[drr]_{\cong} && \Hom(Y'', \pi_{n+1}(S^{n})) \ar[d]_{\cong}^{(\pi|_{Y'})^{*}} \\
&& \Hom(Y', \pi_{n+1}(S^{n}))}$$
commutes. This proves the proposition. 
\end{proof}

It follows from the previous proposition that if $(y_{1}, \dots, y_{g})$ is a unimodular list of elements in $Y$ spanning a free direct summand of rank $g$ such that $\alpha(y_{i}) = 0$ for all $i$, then there exist \textit{unique} bases,
$(x_{1}, \dots, x_{g})$ and $(z_{1}, \dots, z_{g})$
for $X$ and $\text{Torsion}(Y)$ respectively such that the lists
$$(x_{1}, \dots, x_{g}) \quad \text{and} \quad (y_{1}, \dots, y_{g}, z_{1}, \dots, z_{g})$$
satisfy conditions i., ii., and iii. in Definition \ref{defn: wall pairing}. This observation leads to the following proposition: 
\begin{proposition} \label{prop: uniqueness of rho} Let $(X, \; Y, \; \lambda, \; q, \; \alpha)$ be a Wall pairing of rank $g$. There exists a unique map 
$$
\rho: X \longrightarrow \mbox{{\rm Torsion}}(Y) 
$$
that satisfies,
\begin{equation} \label{eq: rho map} q(\rho(x), \; y) \; = \; \lambda(x, y)\circ\rho_{n+1} \quad \text{for all $(x, y) \in X\times Y$} \end{equation}
where $\rho_{n+1} \in \pi_{n+1}(S^{n})$ is the standard generator. 
\end{proposition}
 \begin{proof} 
 Let $(y_{1}, \dots, y_{g})$ be a unimodular list of elements in $Y$ spanning a free direct summand of rank $g$ such that $\alpha(y_{i}) = 0$ for all $i$. Then let 
 $(x_{1}, \dots, x_{g})$ and $(z_{1}, \dots, z_{g})$
 be the unique bases of $X$ and $\text{Torsion}(Y)$ so that 
 $$(x_{1}, \dots, x_{g}) \quad \text{and} \quad (y_{1}, \dots, y_{g}, z_{1}, \dots, z_{g})$$
satisfy conditions i., ii., and iii. in Definition \ref{defn: wall pairing}. We define $\rho: X \longrightarrow Y$ to be the homomorphism determined by 
$x_{i} \mapsto z_{i}$ for $i = 1, \dots, g$. Clearly with this definition $\rho$ satisfies (\ref{eq: rho map}). Uniqueness of the map $\rho$ follows from uniqueness of the complimentary bases $(x_{1}, \dots, x_{g})$ and $(z_{1}, \dots, z_{g})$ to the unimodular sequence $(y_{1}, \dots, y_{g})$. \end{proof}
 
 For the Wall pairing $(\pi_{n}(W_{g,1}), \; \pi_{n+1}(W_{g,1}),\;  \lambda_{n, n+1},\;  \lambda_{n+1, n+1},\; \alpha_{n+1})$, the map $\rho$ is given by the formula,
 $x \mapsto x\circ\rho_{n+1}$ for $x \in \pi_{n}(W_{g,1})$. This is indeed a homomorphism since when $n \geq 4$, the homotopy group $\pi_{n+1}(W_{g,1})$ is in the stable range. 
 
 This map $\rho: X \longrightarrow \text{Torsion}(Y)$ will prove to be useful latter on when working with \text{Wall pairings}. 

\section{High Connectivity of $K_{\bullet}(W_{g,1})$} \label{High connectivity}
In this section we prove that the geometric realization $|K_{\bullet}(W_{g,1})|$ is highly connected. We do so by comparing it to a highly connected simplicial complex $K^{\pi}(W_{g,1})$ modeled on the \text{Wall pairing}  $(\pi_{n}(W_{g,1}), \; \pi_{n+1}(W_{g,1}),\;  \lambda_{n, n+1},\;  \lambda_{n+1, n+1},\; \alpha_{n+1})$ described in the previous section.

\begin{defn} \label{defn: wall complex}  Let $(X, \; Y, \; \lambda, \; q, \; \alpha)$ be a Wall pairing. 
We define $K(X, \; Y, \; \lambda, \; q, \; \alpha)$ to be the simplicial complex whose $p$-simplices are given by sets of pairs 
$$\{(x_{0}, y_{0}), \dots, (x_{p}, y_{p})\} \subset  X\times Y$$
such that for $i, j = 0, \dots, p$ the following conditions are satisfied:
\begin{enumerate}
\item[i.] $\lambda(x_{i}, y_{j}) = \delta_{i,j}\cdot \iota_{n}$,
\item[ii.] $q(y_{i}, y_{j}) = 0$, 
\item[iii.] $\alpha(y_{i}) = 0$.
\end{enumerate}
\end{defn}

By the above definition, each vertex $(x, y) \in K(X, \; Y, \; \lambda, \; q, \; \alpha)$ determines an embedding of the rank $1$ Wall pairing determined by the restrictions of 
$\lambda, q$, and $\alpha$ to the submodules
 $\langle x \rangle \subset X$ and  $\langle y, \;  \rho(x) \rangle \subset Y$.
The main Wall pairing that we are interested in is of course 
$$(\pi_{n}(W_{g,1}), \; \pi_{n+1}(W_{g,1}),\;  \lambda_{n, n+1},\;  \lambda_{n+1, n+1},\; \alpha_{n+1}).$$
In order to save space we will denote,
\begin{equation} K^{\pi}(W_{g,1}) \; := \; K(\pi_{n}(W_{g,1}), \; \pi_{n+1}(W_{g,1}),\;  \lambda_{n, n+1},\;  \lambda_{n+1, n+1},\; \alpha_{n+1}). \end{equation}
We view the complex $K^{\pi}(W_{g,1})$ as an algebraic version of the complex $K^{c}(W_{g,1})$. 

To prove our main theorem, we will need the complex $K^{\pi}(W_{g,1})$ to have a certain technical property. Recall the definition from \cite{GRW 12}:
\begin{defn} \label{defn: cohen mac} A simplicial complex $K$ is said to be \textit{weakly Cohen-Macaulay} of dimension $n$ if it is $(n-1)$-connected and the link of any $p$-simplex is $(n-p-2)$-connected. In this case, we write $\omega CM(K) \geq n$. \end{defn} 

There is a general theorem from \cite{GRW 12} regarding weakly Cohen-Macaulay complexes which we will need. We state the theorem here.

\begin{theorem} {\rm (Galatius and Randal-Williams,  \cite[Theorem 2.4]{GRW 12})} \label{thm: cohen macaulay} Let $X$ be a simplicial complex and $f: \partial I^{n} \longrightarrow |X|$ be a map which is simplicial with respect to some PL triangulation of $\partial I^{n}$. Then, if $\omega CM(X) \geq n$, the triangulation extends to a PL triangulation of $I^{n}$, and $f$ extends to a simplicial map $g: I^{n} \longrightarrow |X|$ with the property that $g(\link(v)) \subset \link(g(v))$ for each interior vertex $v \in I^{n} - \partial I^{n}$. In particular, $g$ is simplexwise injective if $f$ is. \end{theorem}

We now state a theorem whose proof we put off until Section \ref{section: The Intersection Complex}. 
\begin{theorem} \label{thm: pi cohen mac} For $g > 3$, $\omega CM(K^{\pi}(W_{g,1})) \geq \frac{1}{2}(g - 1)$. \end{theorem}
It follows from this theorem that the geometric realization $|K^{\pi}(W_{g,1})|$ is $\frac{1}{2}(g - 3)$-connected. 

Recall from the previous section the elements $e \in \pi_{n}(H)$ and $\eta \in \pi_{n+1}(H)$. There is a simplicial map 
\begin{equation} \label{eq: simplicial map} K^{c}(W_{g,1}) \longrightarrow K^{\pi}(W_{g,1}), \quad (\phi, t) \; \mapsto \; (\phi_{*}(e), \; \phi_{*}(\eta)).\end{equation} 
Recall that for any simplex 
 $((\phi_{0}, t_{0}), \dots, (\phi_{p}, t_{p}))$ the cores 
$\phi_{0}(C), \dots, \phi_{p}(C)$ are  disjoint. Thus 
the above map respects adjacencies and is actually a simplicial map.
We will prove that this map is highly connected.
\begin{theorem} \label{lemma: highly connected map} Let $n := \frac{1}{2}[\dim(W_{g,1})-1]$ be greater than or equal to $4$. Then the map $K^{c}(W_{g,1}) \longrightarrow K^{\pi}(W_{g,1})$ from (\ref{eq: simplicial map}) is $\frac{1}{2}(g-3)$-connected. \end{theorem}
\begin{proof} Let $k \leq \frac{1}{2}(g-3)$ and consider a map $f: S^{k} \longrightarrow |K^{c}(W_{g,1})|$, which we may assume is simplicial with respect to some PL triangulation of $S^{k} = \partial I^{k+1}$. By Theorem \ref{thm: pi cohen mac} the composition $\partial I^{k+1} \longrightarrow |K^{c}(W_{g,1})| \longrightarrow |K^{\pi}(W_{g,1})|$ is null homotopic and so extends to a map $g: I^{k+1} \longrightarrow |K^{\pi}(W_{g,1})|$, which we may suppose is simplicial with respect to a PL triangulation of $I^{k+1}$ extending the triangulation of its boundary. 

By Lemma \ref{thm: pi cohen mac}, $\omega CM(K^{\pi})) \geq
\frac{1}{2}(g-1)$ and so by Theorem \ref{thm: cohen macaulay}, as $k
+1 \leq \frac{1}{2}(g-1)$, we can arrange that $g$ is simplexwise
injective on the interior of $I^{k+1}$. We choose a total order on the
interior vertices of $I^{k+1}$ and inductively choose lifts of each
vertex to $K^{c}(W_{g,1})$ which respect adjacencies.

Let $(x, y)$ be a vertex in the image of the interior of $I^{k+1}$ under the map $g$. We have $x \in \pi_{n}(W_{g,1})$ and $y \in \pi_{n+1}(W_{g,1})$ such that 
$$\alpha(x) = 0, \quad \alpha(y) = 0, \quad \text{and} \quad \lambda_{n, n+1}(x, y) = \iota_{n} \in \pi_{n}(S^{n}).$$
Furthermore, the elements $x$ and $y$ can be represented by embeddings of spheres, $S^{n+1}$ and $S^{n}$, each with trivialized normal bundle. Since $\lambda_{n, n+1}(x, y) = \iota_{n}$, application of the Whitney Trick implies that we may assume that these embedded spheres representing $x$ and $y$ intersect each other transversally at exactly one point. Using a general position argument, we may then assume that $x$ and $y$ are represented by embeddings,
$$j_{x}: S^{n+1}\times D^{n} \longrightarrow W_{g,1} \quad \text{and} \quad j_{y}: D^{n+1}\times S^{n} \longrightarrow W_{g,1}$$
such that 
$$j_{x}(S^{n+1}\times D^{n}) \cap j_{y}(D^{n+1}\times S^{n}) \; = \; U\times V$$
where $U \subset j_{x}(S^{n+1}\times\{0\})$ and $V \subset j_{y}(\{0\}\times S^{n})$ are submanifolds (embedded as closed subsets), diffeomorphic to $D^{n+1}$ and $D^{n}$ respectively. It can be easily checked that the push-out 
\begin{equation} \label{eq: pushout} M_{x,y} :=  j_{x}(S^{n+1}\times D^{n}) \bigcup_{U\times V} j_{y}(D^{n+1}\times S^{n}) \end{equation}
is diffeomorphic to $W_{1, 1} = S^{n+1}\times S^{n} - \Int(D^{2n+1})$, after smoothing the corners. The inclusion of $M_{x,y}$ from (\ref{eq: pushout}) determines an embedding $W_{1,1} \hookrightarrow W_{g,1}$. One then chooses an embedded arc connecting 
$M_{x,y}$  
to $\partial W_{g,1}$ whose interior has empty intersection with $M_{x,y}$ and $\partial W_{g,1}$ (this is possible by a general position argument). A thickening of this ark together with the embedding $W_{1,1} \hookrightarrow W_{g,1}$ then determines an embedding $\phi: H \longrightarrow W_{g,1}$ and thus a lift of the vertex $(x, y)$ to a vertex in $K^{c}(W_{g,1})$. This takes care of the zeroth step of the induction. 

Now suppose that $(x, y)$ is an interior vertex and let $(x_{1}, y_{1}), \dots, (x_{k}, y_{k})$ be the vertices adjacent to $(x, y)$ that have already been lifted and denote by $(\phi_{1}, t_{1}), \dots, (\phi_{k}, t_{k})$ their lifts. Note that the set $\{(x_{1}, y_{1}), \dots, (x_{k}, y_{k})\}$ is not necessarily a simplex in $K^{\pi}(W_{g,1})$. However, $\{(x, y), (x_{j}, y_{j})\}$ is a $1$-simplex in $K^{\pi}(W_{g, 1})$ for $j = 1, \dots, k$. Now, 
$$\lambda_{n, n+1}(x, y_{j}) = 0, \quad \lambda_{n, n+1}(x_{j}, y) = 0, \quad \text{and} \quad \lambda_{n+1, n+1}(y, y_{j}) = 0 \quad \text{for $j = 1, \dots, k$}$$
and so we may apply the Whitney Trick and its generalization Proposition \ref{prop: Generalized Whitney Trick} inductively, to choose embeddings 
$$j_{x}: S^{n} \longrightarrow W_{g,1} \quad \text{and} \quad j_{y}: S^{n+1} \longrightarrow W_{g, 1}$$
representing the classes $x$ and $y$ such that $j_{x}(S^{n})$ and $j_{y}(S^{n+1})$ are disjoint form the \textit{cores} $\phi_{j}(C)$ (recall Section \ref{The Complex of Embedded Handles}),  for all $j \in \{1, \dots, k\}$. Furthermore, since $\lambda_{n, n+1}(x, y) = \iota_{n}$, we may by application of the Whitney Trick assume further that $j_{x}$ and $j_{y}$ are such that $j_{x}(S^{n})$ and $j_{y}(S^{n+1})$ intersect transversally at just one point. We then carry out the plumbing construction employed in the zeroth step of the induction on the embeddings $j_{x}$ and $j_{y}$ to obtain an embedding 
$\phi: H \longrightarrow W_{g,1}$ such that $\phi(C) \cap \phi_{j}(C) = \emptyset$ for all $j \in \{1, \dots, k\}$. This completes the induction and proves the theorem. 
\end{proof}

It follows from the above theorem that $|K^{c}(W_{g,1})|$ is $\frac{1}{2}(g-3)$-connected. Now, since every simplex of the semi-simplicial set $K^{\delta}_{\bullet}(W_{g,1})$ is determined entirely by its vertices, there is a homeomorphism $|K^{c}(W_{g,1})| \cong |K^{\delta}_{\bullet}(W_{g,1})|$, and thus $|K^{\delta}_{\bullet}(W_{g,1})|$ is $\frac{1}{2}(g-3)$-connected as well.

We now consider the natural map $K^{\delta}_{\bullet}(W_{g,1}) \longrightarrow K_{\bullet}(W_{g,1})$. 

\begin{lemma} The semi-simplical space $K_{\bullet}(W_{g,1})$ is $\frac{1}{2}(g-3)$-connected. \end{lemma} 
\begin{proof} This is proven in exactly the same way as \cite[Theorem 4.6]{GRW 12}.  The theorem is based entirely on simplicial techniques and has no dependence on the structure of the manifolds present. \end{proof}

We now derive two important corollaries of the previous lemma. These are both proven in the same way as \cite[Corollary 4.4]{GRW 12} and \cite[Corollary 4.5]{GRW 12}. 

\begin{corollary} \label{cor: transitivity} (Transitivity). Suppose that $g \geq 3$ and let $\phi_{0}, \phi_{1}: H \hookrightarrow W_{g,1}$ be embeddings.  Then there is a diffeomorphism $f$ of $W_{g,1}$ which is isotopic to the identity on the boundary and such that $\phi_{1} = f \circ \phi_{0}$. \end{corollary}

\begin{proof} Suppose first that $\phi_{0}$ and $\phi_{1}$ are disjoint. Let $V$ denote the closure of a regular neighborhood of $\phi_{0}(H)\cup \phi_{1}(H) \cup \partial W_{g,1}$. The submanifold $V$ is abstractly diffeomorphic to $W_{2, 2}$  (which is by definition $(\#^{2}S^{n+1}\times S^{n}) - [\Int(D^{2n+1} \sqcup \Int(D^{2n+1})]$) and has two standard copies of $H$ embedded in it, both connected to the first component of the boundary. With this assumption that $\phi_{0}$ and $\phi_{1}$ are disjoint, it is enough to find a diffeomorphism of $W_{2, 2}$ which is the identity on the first boundary component, is isotopic to the identity on the second boundary component, and sends the first embedded copy of $H$ to the second. 

We give an explicit construction of such a diffeomorphism. For what follows we use the standard model of $S^{2n} \subset \R^{2n+1}$ given by 
$$S^{2n} = \{(x_{1}, \dots, x_{2n+1}) \in \R^{2n+1} \; | \; x_{1}^{2} + \cdots + x_{2n+1}^{2} = 1\}.$$ 
We equip $[0,1]\times S^{2n}$ with the metric obtained by taking the product of the Euclidean metric on $[0,1]$, and the metric on $S^{2n}$ induced from the Euclidean metric on $\R^{2n+1}$. 
Now, let $\Gamma \in SO(2n+1)$ be the diagonal matrix with entries $(-1, -1, 1, \dots, 1)$. Let $d_{0}: D^{2n+1} \rightarrow [0,1]\times S^{2n}$ be the embedding of a disk of radius $\frac{1}{4}$ centered at the point $(\tfrac{1}{2}, (1, 0, \dots, 0)) \in [0,1]\times S^{2n}$. We then set $d_{1} := (Id_{[0,1]}\times\Gamma)\circ d_{0}$. It is immediate that $d_{1}$ is an embedding of a disk of radius $\frac{1}{4}$ centered at the point $(\tfrac{1}{2}, (-1, 0, \dots, 0)) \in [0,1]\times S^{2n}$. We form a manifold denoted by $M_{2, 2}$, by connect summing two copies of $S^{n}\times S^{n+1}$ to $[0,1]\times S^{2n}$ at the disks $d_{0}(D^{2n+1})$ and $d_{1}(D^{2n+1})$,
$$M_{2, 2} = (S^{n}\times S^{n+1})\#_{d_{0}}([0, 1]\times S^{2n})\#_{d_{1}}(S^{n}\times S^{n+1}).$$
Clearly $M_{2, 2}$ is abstractly diffeomorphic to $W_{2, 2}$. There is a diffeomorphism $\varphi$ of $M_{2,2}$ which is equal to $Id_{[0,1]}\times\Gamma$ on $[0,1]\times S^{2n}\setminus(d_{0}(D^{2n+1})\cup d_{1}(D^{2n+1}))$, and interchanges the two copies of $S^{n}\times S^{n+1} - \text{int}(D^{2n+1})$. There is an embedded copy of $H$ given by union the first copy of $S^{n}\times S^{n+1}-\text{int}(D^{2n+1})$ with a thickening of the arc $[\frac{1}{2}, 1]\times \{(1, 0, \dots, 0)\}$, and its image under $\varphi$ gives another disjoint embedded copy of $H$. The diffeomorphism $\varphi$, by construction interchanges the two embedded copies of $H$, however $\varphi$ is not the identity map on the boundary component $\{0\}\times S^{2n}$. To fix this, we replace $Id_{[0,1]}\times\Gamma$ in the above construction with a function of the form 
$$(t, x) \; \mapsto \; (t, \gamma(t)\cdot x) \quad \text{for $(t, x) \in [0,1]\times S^{2n}$}$$
where $\gamma: [0,1]\rightarrow SO(2n+1)$ is a path in $SO(2n+1)$ which is the identity on $[0, \epsilon]$ and $\Gamma$ on $[2\epsilon, 1]$ for $0 < \epsilon < \frac{1}{8}$. 

Now suppose that $\phi_{0}$ and $\phi_{1}$ are not disjoint. We put an equivalence relation on the vertices of $K^{c}(W_{g,1})$. We define two vertices, $(\phi_{0}, t_{0})$ and $(\phi_{1}, t_{1})$ to be equivalent if there exists a diffeomorphism $f$ of $W_{g,1}$, isotopic to the identity on the boundary, such that $f \circ \phi_{0} = \phi_{1}$. The previous paragraph implies that two such vertices are equivalent if $\phi_{0}$ and $\phi_{1}$ are disjoint, or equivalently if the set $\{(\phi_{0}, t_{0}), (\phi_{1}, t_{1})\}$ is a $1$-simplex in the complex $K^{c}(W_{g,1})$. If $g \geq 3$, Theorem \ref{lemma: highly connected map} implies that $|K^{c}(W_{g,1})|$ is path-connected, hence any two vertices are connected via a zig-zag of $1$-simplices. It follows that any two vertices of $K^{c}(W_{g,1})$ are equivalent by our equivalence relation. 
\end{proof} 

\begin{corollary} (Cancelation) \label{cor: cancellation}  Let $M$ be a $(2n+1)$-dimensional manifold with boundary parametrized by $S^{2n}$, and suppose there is a diffeomorphism 
$$\varphi: M\# (S^{n}\times S^{n+1}) \longrightarrow W_{g+1, 1}$$
which is the identity on the boundary. Then if $g \geq 3$ there is a diffeomorphism of $M$ with $W_{g,1}$ which is the identity on the boundary, with respect to the parametrization by $S^{2n}$.  
\end{corollary}

\begin{proof} Recall from Section \ref{The Complex of Embedded Handles} that $(H = [0,1]\times D^{2n})\cup_{\beta} W_{1,1}$ where $\beta: \{1\}\times D^{2n} \rightarrow \partial W_{1,1}$ is an embedding. 
Choose an embedding 
$$\gamma: H \longrightarrow W_{g+1,1}$$
which satisfies $\gamma(\{0\}\times D^{2n}) \subset \partial W_{g+1,1}$,
and such that 
there exists a diffeomorphism 
$$\text{Cl}(W_{g+1,1} - \gamma(H)) \cong W_{g,1}$$
where $\text{Cl}$ denotes the closure.  
Notice that since $W_{g,1}$ is diffeomorphic to the boundary-connect-sum of $g+1$ copies of $H$, one can simply choose $\gamma$ to be the embedding given by inclusion of the last factor of $H$ in the boundary-connect-sum decomposition. 

From the connect sum $M\# (S^{n}\times S^{n+1})$, there is a cannonical embedding 
$$j: W_{1,1} \hookrightarrow M\# (S^{n}\times S^{n+1}).$$ 
By choosing a thickening of an arc which connects $j(\partial W_{1,1})$ to the boundary of the manifold 
$M\# (S^{n}\times S^{n+1})$, we obtain an extension of $j: W_{1,1} \hookrightarrow M\# (S^{n}\times S^{n+1})$ to an embedding
$$f: H \longrightarrow M\#(S^{n+1}\times S^{n})$$
with $f(\{0\}\times D^{2n}) \subset \partial(M\#(S^{n+1}\times S^{n}))$. 
It follows from the construction of $f$, together with the fact that the boundary of $M\#(S^{n+1}\times S^{n})$ is diffeomorphic to $S^{2n}$, that there is a diffeomorphism
$$\text{Cl}(M\# (S^{n}\times S^{n+1}) - f(H)) \; \cong \; M - \Int(D^{2n+1}).$$ 
With 
$\varphi: M\# (S^{n}\times S^{n+1}) \longrightarrow W_{g+1, 1}$
the diffeomorphism from the statement of the corollary, we have two embeddings
$$\varphi\circ f: H \longrightarrow W_{g+1, 1} \quad \text{and} \quad \gamma: H \longrightarrow W_{g+1, 1}.$$
By Corollary \ref{cor: transitivity} there is a diffeomorphism $q$ of $W_{g+1, 1}$ isotopic to the identity on the boundary so that $\gamma = q\circ \varphi\circ f$. We then obtain a diffeomorphism 
$$q\circ \varphi: M\# (S^{n}\times S^{n+1}) \; \longrightarrow \; W_{g+1, 1}$$
such that the composition 
$$\gamma^{-1} \circ q\circ \varphi \circ f: H \longrightarrow H$$
is the identity map. In particular, it maps the submanifold  $W_{1,1} \subset H$ to itself identically. 
So then, after removing the interiors of $f(W_{1,1})$ and $\gamma(W_{1,1})$, the restriction of $q\circ \varphi$ yields a diffeomorphism, 
$$\xymatrix{
M - \Int(D^{2n+1}) \ar[rr]^{\cong} && W_{g,1} - \Int(D^{2n+1})
}$$
which is equal to the identity on the new boundary component corresponding to the removed disk, 
and is isotopic to the identity on the old boundary component  (with respect to the parametrization by $S^{2n}$). We may then fill in the new boundary component with a disk and use isotopy extension to make the diffeomorphism be the identity, with respect to the parametrization by $S^{2n}$, on the remaining boundary component. 
\end{proof}

In the final section we will need the following modification of $K_{\bullet}(W_{g,1})$. 

\begin{defn} Let $\bar{K}_{\bullet}(W_{g,1}) \subset K_{\bullet}(W_{g,1})$ be the sub-semi-simplicial space consisting of all simplices $((\phi_{0}, t_{0}), \dots, (\phi_{p}, t_{p}))$ such that the intersections $\phi_{i}(H) \cap \phi_{j}(H)$ are empty for all $i, j \in \{0, \dots, p\}$. \end{defn}

\begin{corollary} \label{connectivity of K- bar} The space $|\bar{K}_{\bullet}(W_{g,1})|$ is $\frac{1}{2}(g - 3)$-connected. \end{corollary}
\begin{proof} Precomposing with the isotopy $\rho_{t}$, any tuple of embeddings with disjoint cores eventually becomes disjoint. It follows that the inclusion $\bar{K}_{\bullet}(W_{g,1}) \hookrightarrow K_{\bullet}(W_{g,1})$ is a level-wise weak equivalence. \end{proof}

\section{High-Connectivity of the Complex $K^{\pi}(W_{g,1})$} \label{section: The Intersection Complex}
In this section we prove Theorem \ref{thm: pi cohen mac}.
This theorem is about the simplicial complex $K^{\pi}(W_{g,1})$. However, it will be convenient for our purposes to work with partially ordered sets instead of simplicial complexes. 

We review the relationship between simplicial complexes, semi-simplicial sets, and posets. Let $K$ be a simplicial complex. Associated to $K$ is the semi-simplical set $K_{\bullet}$ whose $p$-simplices are the injective simplicial maps $\triangle^{p} \rightarrow K$, i.e. ordered $(p+1)$-tuples of vertices in $K$ spanning a $p$-simplex. There is a natural surjection $|K_{\bullet}| \rightarrow |K|$, and any choice of total order on the set of vertices of $K$ induces a section $|K| \rightarrow |K_{\bullet}|$. In particular, $|K|$ is at least as connected as $|K_{\bullet}|$. 
Associated to the semi-simplicial set $K_{\bullet}$ is a poset denoted by $\widehat{K}_{\bullet}$ whose elements are the simplices of $K_{\bullet}$ and the ordering is given by setting 
$\sigma \leq \alpha \quad \text{if \; $d_{i_{1}}\circ \cdots \circ d_{i_{k}}(\alpha) \; = \; \sigma$}$ 
for some sequence $i_{1}, \dots, i_{k}$, i.e. $\sigma \leq \alpha$ if $\sigma$ is a face of $\alpha$. 
We must now define the geometric realization of a poset. 
\begin{defn} \label{defn: geo realization of poset} Let $E$ any poset. We define $|E|$ to be the geometric realization of the semi-simplicial set whose $p$-simplices are the ordered $(p+1)$-tuples $x_{0} < x_{1} < \cdots < x_{p}$ of elements of $E$. We will refer to $|E|$ as the geometric realization of of the poset $E$. \end{defn}
As above, let $\widehat{K}_{\bullet}$ denote the poset associated to the semi-simplicial space $K_{\bullet}$. It is easy to check that the geometric realization $|\widehat{K}_{\bullet}|$ is the \textit{barycentric sub-division} of $|K_{\bullet}|$. In particular, there is a homotopy equivalence
$
|\widehat{K}_{\bullet}| \sim |K_{\bullet}|.
$

From the above discussion it follows that in order to prove that $|K^{\pi}(W_{g,1})|$ is $\frac{1}{2}(g - 3)$-connected, it will suffice to prove that the geometric realization of the associated  poset $\widehat{K}^{\pi}_{\bullet}(W_{g,1})$ is $\frac{1}{2}(g - 3)$-connected. 

\begin{Notation} If $K$ is a simplicial complex we will denote by $K_{\bullet}$ the semi-simplicial complex associated to $K$ and we will denote by $\widehat{K}_{\bullet}$ the poset associated to $K_{\bullet}$. For any poset $E$ we will denote by $|E|$ the geometric realization as defined in Definition \ref{defn: geo realization of poset}. 
\end{Notation} 

\subsection{Some Basic Results About Posets} \label{Poset Basics} 
We now review from \cite{Ch 87} some basic results on posets that we will need in order to prove that the poset $\widehat{K}^{\pi}_{\bullet}(W_{g,1})$ is  $\frac{1}{2}(g - 3)$-connected. We use many of the constructions and results from \cite{Ch 87} and use much of the same notation.

For a poset $E$ and an element $x \in
E$, we denote,
$$\begin{aligned} E_{< x} := \{ y \in E \; | \; y < x \}, &\quad E_{>
    x} := \{ y \in E \; | \; y > x \}, 
\\ E_{\leq x} := \{ y \in E \;
  | \; y \leq x \}, &\quad E_{\geq x} := \{ y \in E \; | \; y \geq x
  \}.
\end{aligned}$$
It is immediate that $E_{< x}, E_{< x}, E_{< x}$, and $E_{< x}$ are all sub-posets of $E$. 
  
For any set $X$, we denote by $\mathcal{O}(X)$ the poset comprised of
non-empty, finite, ordered lists of distinct elements in $X$. The
partial ordering is given by setting,
$$(w_{1}, \dots, w_{j}) \leq (v_{1}, \dots, v_{k})$$
if $(w_{1}, \dots, w_{j})$ is an ordered sub-list of $(v_{1}, \dots, v_{k})$. 
We say that a sub-poset $F \subset \mathcal{O}(X)$ satisfies the \textit{chain condition} if 
\begin{enumerate}
\item[i.] $v \in F$ implies $\mathcal{O}(X)_{<v} \subset F$, and 
\item[ii] $v \in F$ implies every permutation of the sequence $v = (v_{1}, \dots, v_{k})$ is in $F$. 
\end{enumerate}

If $x < w$ in $\mathcal{O}(X)$, let $w\setminus x$ denote the subsequence of $w$ complimentary to $x$. 
Let $F \subseteq \mathcal{O}(X)$ be a sub-poset and let $v := (v_{1}, \dots, v_{k})$ be an element of $\mathcal{O}(X)$. We define
$$\begin{aligned}
F_{v} \; &:= \; \{(w_{1}, \dots, w_{j}) \in F \; | \; (w_{1}, \dots, w_{j}, v_{1}, \dots, v_{k}) \in F\}\\
Z^{n}F \;&:= \; \{w \in \mathcal{O}(X \sqcup\{z_{1}, \dots, z_{n}\}) \; | \; (z_{1}, \dots, z_{n}) < w \; \text{and} \; w \setminus z \in F\}.
\end{aligned}$$
The posets $F_{v}$ and $Z^{n}F$ arrise 
when studying the \textit{link} of a vertex in $|F|$. In particular, the link of $v$ in $|F|$ is the geometric realization of the poset 
$$\text{link}_{F}(v) \; = \; \{w \in F \; | \; v < w \; \text{or} \; v > w\} \; = \; F_{<v}\star F_{>v}$$
where $\star$ denotes the join.
If $F$ satisfies the chain condition, then $|F_{<v}|$ is homeomorphic to a $(k-2)$-sphere and $F_{>v}$ is isomorphic to $Z^{k}(F_{v})$, with isomorphism given by sending $v_{i}$ to $z_{i}$. Hence there is a homeomorphism of geometric realizations
\begin{equation} \label{eq: link homeomorphism} |\text{link}_{F}(v)| \; \cong \; \Sigma^{k-1}|Z^{k}F_{v}|. \end{equation}
For a non-empty set $S$, define
$$F\langle S \rangle \; = \; \{((w_{1}, s_{1}), \dots, (w_{j}, s_{j})) \in \mathcal{O}(X\times S) \; | \; (w_{1}, \dots, w_{j}) \in F\}.$$
If $v = (v_{1}, \dots, v_{k}) \in \mathcal{O}(X)$, let 
$|v| = k = (\text{length of $v$}).$
We now state several technical lemmas. See \cite{Ch 87} and \cite{Ch 84} for details. For the following we assume that $F \subset \mathcal{O}(X)$ satisfies the chain condition. 

\begin{lemma} \label{lemma: null homotopy} For any $v \in F$, the inclusion $|F_{v}| \hookrightarrow |F|$ is nullhomotopic. \end{lemma}

\begin{lemma}  \label{lemma: homotopy isomorphisms} Suppose $d$ is an integer such that for any $v \in F$, $F_{v}$ is $(d-|v|)$-connected and $F$ is $\text{min}(1, d-1)$-connected. Let $s_{0} \in S$. Then the maps,
$$\begin{aligned}
i_{n}: F \longrightarrow Z^{n}F, \quad & (v_{1}, \dots, v_{k}) \mapsto (v_{1}, \dots, v_{k}, z_{1}, \dots, z_{n}),\\
(s_{0})_{*}: F \longrightarrow F\langle S\rangle, \quad & (v_{1}, \dots, v_{k}) \mapsto ((v_{1}, s_{0}), \dots, (v_{k}, s_{0})), \\
i_{n}\circ (s_{0})_{*}: F \longrightarrow Z^{n}(F\langle S\rangle) & 
\end{aligned}$$
induce isomorphisms on homotopy groups $\pi_{j}$ for $j \leq d$. \end{lemma}

\begin{lemma}  \label{lemma: intersection connectivity} Suppose $Y \subseteq X$. If $d$ is an integer such that $F\cap \mathcal{O}(Y)$ is $d$-connected and $F_{v} \cap\mathcal{O}(Y)$ is $(d - |v|)$-connected for all $v \in F$, then $F$ is $d$-connected. 
\end{lemma}

\subsection{The Main Theorem} Now we proceed to prove Theorem \ref{thm: pi cohen mac}. The first step here is to construct a poset model for the complex $K^{\pi}_{\bullet}(W_{g,1})$.
\begin{defn} Let $(X, \; Y,\; \lambda,\; q,\; \alpha)$ be a Wall pairing.   
We define 
$$L(X, \; Y,\; \lambda,\; q,\; \alpha) \subset \mathcal{O}(X\times Y)$$
 to be the sub-poset of $\mathcal{O}(X\times Y)$
consisting of sequences
$((x_{1}, y_{1}), \dots, (x_{k}, y_{k}))$
such that for all $i, j \in \{1, \dots, k\}$:
\begin{enumerate}
\item[i.] $\lambda(x_{i}, y_{j}) = \delta_{i,j}\cdot \iota_{n}$,
\item[ii.] $q(y_{i}, y_{j}) = 0$, 
\item[iii.] $\alpha(y_{i}) = 0$.
\end{enumerate}
\end{defn}

When the context is clear we will suppress $\lambda, q,$ and $\alpha$ from the notation and denote 
$$L(X, \; Y) \; := \; L(X, \; Y,\; \lambda,\; q,\; \alpha).$$

By the discussion from the beginning of Section \ref{section: The Intersection Complex} and from the fact that all rank $g$ \text{Wall pairings} are isomorphic, it is clear that if $(X, \; Y,\; \lambda,\; q,\; \alpha)$ is a rank $g$ Wall pairing, then the geometric realization $|L(X, Y)|$ is homeomorphic to the barycentric subdivision of the space $|K^{\pi}_{\bullet}(W_{g,1})|$ (recall, $K^{\pi}_{\bullet}(W_{g,1})$ is the semi-simplicial space determined by the simplicial complex $K^{\pi}(W_{g,1})$). 

The main step in proving Theorem \ref{thm: pi cohen mac} will be to prove the following:
 \begin{theorem} \label{theorem: wall complex connectivity} Let $(X, \; Y,\;  \lambda, \; q,  \; \alpha)$ be a Wall pairing  of rank $g$. Then the geometric realization $|L(X, Y)|$ is $\frac{1}{2}(g-3)$-connected. \end{theorem}
 
 It follows from this theorem that the space $|K^{\pi}(W_{g,1})|$ is $\frac{1}{2}(g-3)$-connected. We will need some auxiliary constructions. 

\begin{defn} For any free $\Z$-module $V$, Let $U(V) \subset \mathcal{O}(V)$ denote the poset consisting of sequences $(x_{1}, \dots, x_{k})$ of elements of $V$ which are unimodular, i.e. the sub $\Z$-module spanned by $(x_{1}, \dots, x_{k})$ splits as a direct summand of $V$. \end{defn}

We will need to use the following important theorem from \cite[Section 2]{Ka 80}. 
\begin{theorem} {\rm (W. Van der Kallen,  \cite{Ka 80})}   
\label{theorem: unimodular connectivity} 
Let $V$ be a free $\Z$-module of rank $g$. Then the geometric
realization $|U(V)|$ is $(g - 3)$-connected. Furthermore, if $v \in
U(V)$ then the geometric realization $|U(V)_{v}|$ is $(g - 3 -
|v|)$-connected. 
\end{theorem} 

This result will play a crucial role in
our proof of Theorem \ref{theorem: wall complex connectivity}. We now
define new poset which will allow us to compare the poset $L(X, Y)$ to
the poset $U(X)$ for any Wall pairing $(X, Y, \lambda, q, \alpha)$.

\begin{defn} For a Wall-pairing $(X, Y, \lambda, q, \alpha)$,  we define 
$$M(X, Y) := M(X, Y, \lambda, q, \alpha)$$ to 
be the sub-poset of $\mathcal{O}(X\times Y)$ consisting of sequences
  $((x_{1}, y_{1}), \dots, (x_{k}, y_{k}))$
  subject to the following conditions:
  \begin{enumerate}
  \item[i.] the list $(x_{1}, \dots, x_{k})$ is an element of $U(X)$, 
  \item[ii.] $\alpha(y_{i}) = 0 $ for all $i \in \{1, \dots, k\}$,
  \item[iii.] for each $i$, either $y_{i} = 0$ or $\lambda(x_{j}, y_{i}) = \delta_{i, j}\cdot\iota_{n}$ for all $j \in \{1, \dots, k\}$,
  \item[iv.] $q(y_{i}, y_{j}) = 0$ for all $i, j \in \{1, \dots, k\}$.
  \end{enumerate}
  \end{defn}
  
   We will need to use one more important result about Wall pairings. For a Wall pairing $(X,\; Y, \; \lambda, \; q, \; \alpha)$ of rank $g$, let 
  $V \subset X$ and $W \subset Y$ be free sub-$\Z$-modules of rank $k > 0$ that each split as direct summands of $X$ and $Y$,
  such that the restriction 
  $$\lambda\vert_{V\times W}: V\times W \longrightarrow \pi_{n}(S^{n})$$
  is a perfect pairing. Denote, 
  \begin{equation} \label{eq: compliment notation}
  \begin{aligned}
  V^{\perp} \; &= \; \{y \in Y \; | \; \lambda(v, y) = 0 \; \text{for all $v \in V$}\}, \\
  W^{\perp} \; &= \; \{x \in X \; | \; \lambda(x, w) = 0 \; \text{for all $w \in W$}\}, \\
  \widehat{W} \; &= \; \{y \in W \; | \; q(y, w) = 0 \; \text{for all $w \in W$}\}.
  \end{aligned}
  \end{equation}
  Denote by $\lambda'$, $q'$, and $\alpha'$ the restrictions of the maps $\lambda, q,$ and $\alpha$ to $W^{\perp}$ and $V^{\perp}\cap \widehat{W}$. 
  \begin{lemma} \label{lemma: compliment module} Let $(X,\; Y, \; \lambda, \; q, \; \alpha)$, $V \subset X$, and $W \subset Y$ be as above. The tuple
  $$(W^{\perp}, \; V^{\perp}\cap \widehat{W}, \; \lambda', \; q', \; \alpha')$$
  is a \text{Wall pairing} of rank $g - k$.
  \end{lemma}
 \begin{proof} In order to prove the lemma it will suffice to find a 
$\Z$-basis $(x_{1}, \dots, x_{g-k})$ of $W^{\perp}$ and a unimodular
   sequence $(y_{1}, \dots, y_{g-k})$ of elements of
   $V^{\perp}\cap\widehat{W}$ such that for all $i, j \in \{1, \dots,
   g- k\}$ the following conditions are met:
 \begin{enumerate} 
 \item[(a)] $\lambda(x_{i}, y_{j}) \; = \; \delta_{i,
   j}\cdot\iota_{n}$,
 \item[(b)] $q(y_{i}, y_{j}) = 0$,
 \item[(c)] $\alpha(y_{i}) = 0$, 
 \item[(d)] $\lambda(v, y_{i}) = 0$ \and $q(y_{i}, w) = 0$ for all $v \in V$, $w \in W$. 
 \end{enumerate}
 We will then show that,
  $$\langle y_{1}, \dots, y_{g-k}\rangle\oplus \langle \rho(x_{1}), \dots, \rho(x_{g-k}) \rangle \; = \; (V^{\perp}\cap \widehat{W}),$$ 
  where recall from Proposition \ref{prop: uniqueness of rho} that 
  $\rho: X \longrightarrow \text{Torsion}(Y)$
 is the unique homomorphism such that 
 $q(\rho(x), y) \; = \; \lambda(x, y)\circ\rho_{n+1}$ for all $(x, y) \in X\times Y$. 
 This implies
  $$(x_{1}, \dots, x_{g-k}) \quad \text{and} \quad (y_{1}, \dots,
 y_{g-k}, \rho(x_{1}), \dots, \rho(x_{g-k}))$$ are bases of
 $W^{\perp}$ and $V^{\perp}\cap \widehat{W}$ which satisfy conditions
 i., ii., and iii. of Definition \ref{defn: wall pairing}, thus
 implying that the tuple $(W^{\perp}, \; V^{\perp}\cap \widehat{W}, \;
 \lambda', \; q', \; \alpha')$ is indeed a Wall pairing.

 Since $(X,Y, \lambda, q, \alpha)$ is a Wall pairing, 
it follows that there exists a free
 sub-$\Z$-module $Y_{0} \subset Y$ with  $Y \; = \;
 Y_{0}\oplus\text{Torsion}(Y)$, such that
$$\alpha(Y_{0}) = 0 \quad  \text{and} \quad  q(Y_{0}\times Y_{0}) = 0.$$

Since $\lambda': V\times W \longrightarrow \pi_{n}(S^{n})$, the
restriction of $\lambda$ as above, is by hypothesis a perfect pairing,
we can choose $\Z$-bases $(v_{1}, \dots, v_{k})$ and $(w_{1}, \dots,
w_{k})$ of $V$ and $W$ respectively such that $\lambda'(v_{i}, w_{j}) =
\delta_{i, j}\cdot\iota_{n}$ for all $i, j \in \{1, \dots, k\}$.
 
Denote by $\text{proj}_{Y_{0}}: Y \rightarrow Y_{0}$ the projection
induced by the splitting $Y \; = \; Y_{0}\oplus\text{Torsion}(Y)$.
Now, the restriction of $\lambda$ to $V\times\text{proj}_{Y_{0}}(W)$
is a perfect pairing. Since $\text{proj}_{Y_{0}}(W)^{\perp} =
W^{\perp}$, the restriction of $\lambda$ to $W^{\perp}\times
\text{proj}_{Y_{0}}(V^{\perp})$ is a perfect pairing as well. So, we
can choose $\Z$-bases
 $$(x_{0}, \dots, x_{g-k}) \quad \text{and} \quad (y_{1}, \dots, y_{g-k}),$$
 of $W^{\perp}$ and $\text{proj}_{Y_{0}}(V^{\perp})$ respectively, such that 
 $$\lambda(x_{i}, y_{j}) = \delta_{i,j}\cdot\iota_{n} \quad  \text{for all $i, j \in \{1, \dots, g-k\}$}.$$ 
 Since $\alpha(Y_{0}) = 0$ and $q(Y_{0}\times Y_{0}) = 0$, it follows that \; $q(y_{i}, y_{j}) = 0$ and $\alpha(y_{i}) = 0$ for all $i, j \in \{1, \dots, g-k\}$. By construction $\lambda(v, y_{i}) = 0$ for all $v \in V$ and $i \in \{1, \dots, g-k\}$. 
 However, it may happen that $q(y_{i}, w) = \rho_{n+1} \in \pi_{n+1}(S^{n})$ for some $i$ and $w \in W$ and so the elements $y_{i}$ may not all be in $V^{\perp}\cap\widehat{W}$, and thus condition (d) is not yet satisfied. So, in order to find a unimodular sequence that satisfies conditions (a), (b), (c), and (d), we must alter the elements $y_{i}$. For $i = 1, \dots, g-k$, set 
 $$A_{i} \; = \; \{j \in \{1, \dots, k\} \; | \; q(y_{i}, w_{j}) \neq 0 \}.$$
 We then set,
 $$\hat{y}_{i} \; := \; y_{i} + \sum_{j \in A_{i}}\rho(v_{j}).$$
 Observe that since $y_{i} \in \text{proj}_{Y_{0}}(V^{\perp}) \subset V^{\perp}$, 
 $$q(y_{i}, \rho(v_{j})) \; = \; \lambda(v_{i}, y_{j})\circ\rho_{n+1}
 \; = \; 0 \quad \text{for all $j \in \{1, \dots, k\}$}.$$ Then, since
 $\alpha(y_{i}) = 0$ and $\alpha(\rho(v_{j})) = 0$, the summation
 formula (\ref{eq: alpha summation formula}) for $\alpha$ implies that
 $\alpha(\hat{y}_{i}) = 0$ for all $i \in \{1, \dots, g-k\}$. We then
 compute:
 $$\begin{aligned}
q(\hat{y}_{i}, w_{t}) \; =& \; q(y_{i}, w_{t}) + \sum_{j \in A_{i}}q(\rho(v_{j}), w_{t}) = \\
= \;  q(y_{i}, w_{t}) + \sum_{j \in A_{i}}\lambda(v_{j}, w_{t})\circ\rho_{n+1} \; =&
\; \begin{cases} 
q(y_{i}, w_{t}) + \lambda(v_{t}, w_{t})\circ\rho_{n+1} & \quad \text{if $t \in A_{i}$}\\
0 & \quad \text{if $t \notin A_{i}$}
\end{cases}
\end{aligned}
$$ Here the last equality holds since $\lambda(v_{j}, w_{t}) =
 \delta_{j, t}\cdot\iota_{n}$.  In both cases, $t \in A_{i}$ or $t
 \notin A_{i}$, we have $q(\hat{y}_{i}, w_{t}) = 0$ since if $t \in
 A_{i}$, then $q(y_{i}, w_{t}) = \rho_{n+1}$ and $q(\rho(v_{t}),
 w_{t}) = \lambda(v_{t}, w_{t}) = \rho_{n+1}$ by definition, and
 $\rho_{n+1}$ is an element of order $2$.  It is now clear that the
 sequences $(x_{1}, \dots, x_{g-k})$ and $(\hat{y}_{1}, \dots,
 \hat{y}_{g-k})$ satisfy conditions (a), (b), (c), and (d).

The Lemma will be proven once we show that 
$$V^{\perp}\cap\widehat{W} = \langle \hat{y}_{1}, \dots, \hat{y}_{g-k}\rangle\oplus \langle \rho(x_{1}), \dots, \rho(x_{g-k})\rangle.$$
Clearly, 
$$\langle \hat{y}_{1}, \dots, \hat{y}_{g-k}\rangle\oplus \langle \rho(x_{1}), \dots, \rho(x_{g-k})\rangle \;  \subset \; (V^{\perp}\cap\widehat{W}).$$
Let $z \in \; (V^{\perp}\cap\widehat{W})$. 
It is easy to check that
$$Y \; = \; \langle\hat{y}_{1}, \dots, \hat{y}_{g-k}\rangle \oplus \langle\rho(x_{1}), \dots, \rho(x_{g-k}) \rangle \oplus \rho(V) \oplus W,$$
and so $z$ can be written uniquely as $z = a + b + c + d$ with $a \in \langle\rho(x_{1}), \dots, \rho(x_{r}) \rangle$, \; $b \in \langle\hat{y}_{1}, \dots, \hat{y}_{r}\rangle$, \; $c \in \rho(V)$, and $d \in W$. However, since 
$$\langle \rho(x_{1}), \dots, \rho(x_{g-k}) \rangle \oplus \langle\hat{y}_{1}, \dots, \hat{y}_{g-k}\rangle \; \subset \; V^{\perp}\cap \widehat{W},$$
it follows that $c$ and $d$ must both be zero and thus, $z \in \langle \rho(x_{1}), \dots, \rho(x_{g-k}) \rangle \oplus \langle\hat{y}_{1}, \dots, \hat{y}_{g-k}\rangle$. This proves that,
$$V^{\perp}\cap \widehat{W} \; \subset \; \langle \rho(x_{1}), \dots, \rho(x_{g-k}) \rangle \oplus \langle\hat{y}_{1}, \dots, \hat{y}_{g-k}\rangle.$$
It follows that 
$(W^{\perp}, \; V^{\perp}\cap \widehat{W}, \; \lambda', \; q', \; \alpha|_{W^{\perp}})$
is a Wall pairing of rank $g-k$.
\end{proof}

Notice that the posets $L(X, Y)$, $M(X, Y)$ and $U(X)$ all satisfy the
chain condition (see Section \ref{Poset Basics}). Lemma \ref{lemma:
  compliment module} will be useful to us in the following way. Let
$$(v, w): = ((v_{1}, w_{1}), \dots, (v_{k}, w_{k}))$$ 
be an element of $K(X, Y)$. By definition of $K(X,Y)$ we have that 
$\lambda(v_{i}, w_{j}) = \delta_{i, j}\cdot\iota_{n}$ for all $i, j \in \{1, \dots k\}$. Setting 
$$V := \langle v_{1}, \dots, v_{k} \rangle \quad  \text{and}  \quad W = \langle w_{1}, \dots, w_{k} \rangle,$$ 
and using the notation from (\ref{eq: compliment notation}) we have
\begin{equation} \label{eq: identification of link} L(X, Y)_{(v, w)} \; = \; L(W^{\perp}, V^{\perp}\cap \widehat{W}). \end{equation}

 Notice also that the posets $U(X)$ and $L(X, Y)$ both embed in $M(X, Y)$ as sub-posets. We have,
  \begin{equation} \label{eq: sub poset chain}
  \begin{aligned}
  M(X, Y) \cap \mathcal{O}(X \times \{0\}) \; &= \; U(X) \quad \text{and}\\
  M(X, Y) \cap \mathcal{O}(X \times (Y\setminus \{0\})) \; &= \; L(X, Y).
  \end{aligned}
  \end{equation}
These above identifications (\ref{eq: sub poset chain}) and (\ref{eq:
  identification of link}) will form the basis of the proof of Theorem
\ref{theorem: wall complex connectivity} and of the following lemma.
\begin{lemma} Let $(X, Y, \lambda, q, \alpha)$ be a Wall pairing of rank $g$. Then $|M(X, Y)|$ is $(g - 3)$-connected. \end{lemma}
  \begin{proof} Let $v$ be an element of $M(X, Y) := M(X, Y, \lambda, q, \alpha)$. Re-ordering if necessary, we can write,
  $$v \; = \; ((v_{1}, 0), \dots, (v_{i}, 0), (v_{i+1}, w_{i+1}), \dots, (v_{k}, w_{k}))$$
  where $w_{j} \neq 0$ for $j = i+1, \dots, k$. Set 
  $$V = \langle v_{i+1}, \dots, v_{k} \rangle \quad \text{and} \quad W = \langle w_{i+1}, \dots, w_{k} \rangle.$$
  
  By Lemma \ref{lemma: compliment module}, $(W^{\perp}, \; V^{\perp}\cap \widehat{W}, \; \lambda', \; q', \; \alpha')$ is a Wall pairing of rank $g-k+i$. It can then be easily checked that,
$$M(X, Y)_{v} \; = \; M(W^{\perp}, V^{\perp}\cap \widehat{W})_{((v_{1}, 0), \dots, (v_{i}, 0))}.$$
From this it follows that
$$M(X, Y)_{v}\cap \mathcal{O}(X\times \{0\}) \; \cong \; U(W^{\perp})_{(v_{1}, \dots,  v_{i})},$$
and thus by Theorem \ref{theorem: unimodular connectivity}, $|M(X, Y)_{v}\cap \mathcal{O}(X\times \{0\})|$ is $((g - k + i) - i - 3)$-connected. Since $v$ was arbitrary, it follows from Lemma \ref{lemma: intersection connectivity} that $|M(X, Y)|$ is $(g - 3)$-connected.
\end{proof}
  
\begin{proof}[Proof of Theorem \ref{theorem: wall complex connectivity}] The proof of this theorem is similar to the proof of \cite[Theorem 3.2]{Ch 87}.

We prove this theorem by induction on $g$, the rank of the Wall
pairing. Let $g = 1$. Proving the theorem in this case amounts to
proving that $|L(X, Y)|$ is $-1$-connected, i.e. non-empty. This is
trivial. In a Wall pairing of rank $1$, there always exists, by
definition, a pair $(x, y) \in X\times Y$ with $\lambda(x, y) = 1$ and
$\alpha(y) = 0$ and so the set of zero simplices is non-empty.
  
  Now suppose that the theorem holds for all Wall pairings $(X, Y,
  \lambda, q, \alpha)$ of rank less than $g$, so that for any such
  Wall pairing of rank $k < g$, the associated poset is $\frac{1}{2}(k
  -3)$-connected. Denote $d := \frac{1}{2}(g-3)$. For an element
  $$(v, w) = ((v_{1},w_{1}), \dots, (v_{i}, w_{i})) \in L(X, Y)$$
consider the poset $L(X, Y)_{(v,w)}$. Setting $V = \langle v_{1}, \dots, v_{i} \rangle$ and $W = \langle w_{1}, \dots, w_{i} \rangle$, Lemma \ref{lemma: compliment module} implies that $(W^{\perp}, \; V^{\perp}\cap \widehat{W}, \; \lambda, \;  q, \; \alpha)$ is a Wall pairing of rank $g - i$ (we are using the same notation as in Lemma \ref{lemma: compliment module}). We then have,
$$L(X, Y)_{(v, w)} \; = L(W^{\perp}, V^{\perp}\cap \widehat{W})$$
and this is $(d - \frac{1}{2}i)$-connected by the induction hypothesis, assuming that $i \geq 1$. We may repeat this argument to deduce further that $L(W^{\perp}, V^{\perp}\cap \widehat{W})_{z}$ is $(d - \frac{1}{2}i - |z|)$-connected for any element $z \in  L(W^{\perp}, V^{\perp}\cap \widehat{W})$. 

Recall the construction from Section \ref{Poset Basics}: If $S$ and $R$ are sets and $F \subset \mathcal{O}(R)$ is a sub-poset then $F\langle S\rangle$ is the poset defined by setting
$$F\langle S\rangle \; := \; \{(r_{1}, s_{1}), \dots, (r_{n}, s_{n}) \in \mathcal{O}(R)\times S \; | \; (r_{1}, \dots, r_{n}) \in F\; \}.$$
We will need to consider the poset $L(W^{\perp}, V^{\perp}\cap \widehat{W})\langle V\times \rho(V) \rangle$. 

Since $L(W^{\perp}, V^{\perp}\cap \widehat{W})$ is $(d - \frac{1}{2}i)$-connected, and that $L(W^{\perp}, V^{\perp}\cap \widehat{W})_{z}$ is $(d - \frac{1}{2}i - |z|)$-connected for any element $z \in  L(W^{\perp}, V^{\perp}\cap \widehat{W})$ as established above, 
Lemma \ref{lemma: homotopy isomorphisms} implies that the poset $L(W^{\perp}, V^{\perp}\cap \widehat{W})\langle V\times \rho(V) \rangle$ is $(d - \frac{1}{2}i)$-connected as well. 

Set $v := ((v_{1}, 0), \dots, (v_{i}, 0)) \in M(X, Y)\cap\mathcal{O}(X\times\{0\})$. Let $V = \langle v_{1}, \dots, v_{i} \rangle$ and 
$W = \langle w_{1}, \dots, w_{i} \rangle$ still be as above. We will need the following proposition. 
\begin{proposition} \label{prop: equivalence 1} With $V, \; W$, and $v = (v_{1}, 0), \dots, (v_{i}, 0))$ as above, there is an isomorphism of Posets,
\begin{equation} L(X, Y)\cap M(X, Y)_{v} \; \cong \; L(W^{\perp}, V^{\perp}\cap \widehat{W})\langle V\times \rho(V) \rangle. \end{equation}
In particular, the geometric realization 
 $|L(X, Y)\cap M(X, Y)_{v}|$ is $(d - \frac{1}{2}i)$-connected. 
\end{proposition}
\begin{proof}
We define a map 
\begin{equation} \label{eq: poset iso 2} L(X, Y)\cap M(X, Y)_{v} \longrightarrow L(W^{\perp}, V^{\perp}\cap \widehat{W})\langle V\times \rho(V) \rangle \end{equation}
as follows. Let $((x_{1}, y_{1}), \dots, (x_{k}, y_{k})) \in L(X, Y)\cap M(X, Y)_{v}$. We have,
$$X = V\oplus W^{\perp} \quad \text{and} \quad Y = (V^{\perp}\cap \widehat{W}) \oplus W \oplus \rho(V).$$
Then for each $j$, $x_{j}$ can be written uniquely as the sum,
 $$x_{j} = a_{j} + b_{j} \quad \text{for $a_{j} \in W^{\perp}$ and $b_{j} \in V$}$$
 and $y_{j}$ can be written uniquely as the sum
 $$y_{j} \; = \; c_{j} + d_{j} + e_{j} \quad \text{for $c_{j} \in V^{\perp}\cap\widehat{W}$, $d_{j} \in W$, and $e_{j} \in \rho(V)$.}$$
 However,  since 
$((x_{1}, y_{1}), \dots, (x_{k}, y_{k})) \in L(X, Y)\cap M(X, Y)_{v}$, it follows by definition that $\lambda(v, y_{j}) = 0$ for all $v \in V$ and $i = 1, \dots, k$. Thus, since $\lambda(v_{l}, w_{j}) = \delta_{l,j}\cdot\iota_{n}$ for all $l, j$, it follows that $d_{j} = 0$ for all $j \in \{1, \dots, k\}$.  
We define the map (\ref{eq: poset iso 2}) by the formula,
 $$((x_{1}, y_{1}), \dots, (x_{k}, y_{k})) \; \mapsto \; ((a_{1}, c_{1}), \dots, (a_{k}, c_{k}), \;  (b_{1}, e_{1}), \dots, (b_{k}, e_{k})).$$
 Given such a sequence 
 $$((a_{1}, c_{1}), \dots, (a_{k}, c_{k}), \;  (b_{1}, e_{1}), \dots, (b_{k}, e_{k})) \in L(W^{\perp}, V^{\perp}\cap \widehat{W})\langle V\times \rho(V) \rangle,$$
 it is easy to verify that the sequence 
 $$((a_{1} + b_{1}, c_{1} + e_{1}), \dots, (a_{k} + b_{k}, c_{k} +  e_{k}))$$
 is an element of $ L(X, Y)\cap M(X, Y)_{v}$. We define the inverse map via,
 $$((a_{1}, c_{1}), \dots, (a_{k}, c_{k}), \;  (b_{1}, e_{1}), \dots, (b_{k}, e_{k})) \; \mapsto \; ((a_{1} + b_{1}, c_{1} + e_{1}), \dots, (a_{k} + b_{k}, c_{k} +  e_{k})).$$
 This proves the proposition. 
\end{proof}
We now denote $F := M(X, Y)$ and filter $F$ by the sub-posets,
$$\begin{aligned}
F_{0} \; &= \; \{((x_{1}, y_{1}), \dots, (x_{k}, y_{k})) \in M(X, Y) \; | \; y_{j} \neq 0 \; \text{for some $j$}\},\\
F_{i} \; &= \; F_{i-1}\cup \{((x_{1}, 0), \dots, (x_{i}, 0)) \in M(X, Y)\cap\mathcal{O}(X\times\{0\}) \; \}.\\
\end{aligned}$$
\begin{proposition} \label{prop: deformation retraction}
The inclusion map $j: L(X, Y) \hookrightarrow F_{0}$ induces a
homotopy equivalence of the geometric realizations $|L(X, Y)|$ and
$|F_0|$. Furthermore, the pair $(|F_{0}|, |L(X, Y)|)$ is a deformation
retraction pair. \end{proposition}
\begin{proof} We define a retraction
\begin{equation} \label{eq: def retraction} h: F_{0} \longrightarrow L(X, Y) \end{equation}
by setting $h((x_{1}, y_{1}), \dots, (x_{k}, y_{k}))$ equal to the sublist consisting of those $(x_{l}, y_{l})$ such that $y_{l} \neq 0$. It follows immediately that $h\circ j = Id_{L(X,Y)}$. Now, notice that for all $z \in F_{0}$, we have $j\circ h(z) \leq z$. It follows from this that the induced map $|j\circ h|: |F_{0}| \rightarrow |F_{0}|$ is homotopic to the identity. 
This proves that $|h|: |F_{0}| \rightarrow |L(X, Y)|$ is a deformation-retraction.  
\end{proof}
To obtain the space $|F_{i}|$ from $|F_{i-1}|$ we attach a cone to the
link (in $F_{i-1}$) of each element of the form $v = ((v_{0}, 0),
\dots, (v_{i}, 0)) \in F_{i} \setminus F_{i-1}$ (all elements of
$F_{i} \setminus F_{i-1}$ are of this form). Now it is easy to see
that for any such $v \in F_{i} \setminus F_{i-1}$,
$$\text{link}_{F_{i-1}}(v) \; = \; \text{link}_{F_{i}}(v).$$
By the results from Section \ref{Poset Basics} we have
$$ |\text{link}_{F_{i-1}}(v)| \; = \;  |\text{link}_{F_{i}}(v)| \; \cong \; \Sigma^{i-1}|Z^{i}(F_{i})_{v}|.$$
Now, the poset $(F_{i})_{v}$ is contained in $F_{0}$ and the restriction to $F_{0}$ of the map $h: F_{0} \longrightarrow L(X, Y)$ from the proof of Proposition \ref{prop: deformation retraction}, has the sub-poset 
$$(L(X, Y)\cap F_{v}) \subset L(X, Y)$$ 
as its image.  In a way similar to the proof of Proposition \ref{prop: deformation retraction} it can be checked that, 
$$|h\vert_{(F_{i})_{v}}|: |(F_{i})_{v}| \longrightarrow |L(X, Y)\cap F_{v}|$$ 
and the induced map,
$$|Z^{i}(h\vert_{(F_{i})_{v}})|: |Z^{i}(F_{i})_{v}| \longrightarrow |Z^{i}(L(X, Y)\cap F_{v})|$$
are deformation retractions.
By Proposition \ref{prop: equivalence 1}, $|L(X, Y)\cap F_{v}|$ is $(d - \frac{1}{2}i)$-connected. We then apply Lemma \ref{lemma: homotopy isomorphisms} to get that $|Z^{i}(L(X, Y)\cap F_{v})|$ is $(d - \frac{1}{2}i)$-connected and thus by the above deformation retractions, $\Sigma^{i-1}|Z^{i}(F_{i})_{v}|$ is $(d + \frac{1}{2}i -1)$-connected. Notice that if $i \geq 2$, $|F_{i}|$ is obtained from $|F_{i-1}|$ by attaching cones to $d$-connected spaces. Thus for $j \leq d$, an inductive argument using the Mayer-Vietoris sequence and Van-Kampen's theorem gives,
$$\pi_{j}(F_{1}) = \pi_{j}(F_{2}) = \cdots = \pi_{j}(M(X, Y)) = 0.$$
If $i = 1$, the space $|\text{link}_{F_{1}}(v)|$ is only $(d - \frac{1}{2})$-connected. In this case, 
$$\pi_{j}(L(X,Y)) = \pi_{j}(F_{0}) = \pi_{j}(F_{1}) = 0$$
provided $j < d$, but for $j = d$ we only get a surjection 
$$\bigoplus_{v \in F_{1}\setminus F_{0}}
\pi_{d}(\text{link}_{F_{1}}(v)) \longrightarrow \pi_{d}(F_{0}).$$ The
surjection here follows from the Mayer-Vietoris sequence (or the van
Kampen theorem if $d = 1$) and the fact that $\pi_{d}(F_{1}) = 0$.

Now let $\hat{v} = (v, 0) \in F_{1}\setminus F_{0}$. We can choose $w \in Y$ such that $(v, w) \in L(X, Y)$ and define a map
$$\Psi: L(X,Y)_{(v, w)} \longrightarrow \text{link}_{F_{1}}(\hat{v})$$
by 
$$((x_{1}, y_{y}), \dots, (x_{j}, y_{j})) \;  \mapsto \; ((x_{1}, y_{1}), \dots, (x_{j}, y_{j}), (v, 0)).$$ 
We claim that $\Psi$ induces an isomorphism on $\pi_{d}$. We have a commutative diagram
$$\xymatrix{
L(X, Y)_{(v,w)} \ar[rrr]^{\Psi} \ar[d]^{\cong} &&& \text{link}_{F_{1}}(\hat{v}) \ar[d]^{\cong} \\
L(W^{\perp}, V^{\perp}\cap \widehat{W}) \ar[d]^{l_{1}\circ (0)_{*}} &&& Z^{1}(F_{1})_{\hat{v}} \\
Z^{1}(L(W^{\perp}, V^{\perp}\cap \widehat{W})\langle V\times \rho(V) \rangle \ar[rrr]^{\cong} &&& Z^{1}(L(X, Y)\cap M(X, Y)_{\hat{v}}) \ar[u]^{\cong} 
 }
$$
where $V = \langle v \rangle$, $W = \langle w \rangle$, and $l_{1}\circ (0)_{*}$ is the map defined in Lemma \ref{lemma: homotopy isomorphisms}.
 The map 
 $$l_{1}\circ (0)_{*}: L(W^{\perp}, V^{\perp}\cap \widehat{W}) \longrightarrow Z^{1}(L(W^{\perp}, V^{\perp}\cap \widehat{W})\langle V\times \rho(V) \rangle$$
 induces an isomorphism on $\pi_{d}$ by Lemma \ref{lemma: homotopy isomorphisms}, thus commutativity of the above diagram implies that $\Psi$ does indeed induce an isomorphism on $\pi_{d}$. Combining these observations, we get a commutative diagram,
 $$\xymatrix{
 \bigoplus_{\hat{v} \in F_{1} \setminus F_{0}}\pi_{d}(\text{link}_{F_{1}}(\hat{v})) \ar[rr] && \pi_{d}(F_{0})\\
 \bigoplus_{\hat{v} \in F_{1} \setminus F_{0}}\pi_{d}(L(X, Y)_{(v, w)}) \ar[u]^{\Psi_{*}}_{\cong} \ar[rr] && \pi_{d}(L(X, Y)) \ar[u]^{h_{*}}_{\cong}}$$
 where the bottom horizontal map is induced by the inclusion $L(X, Y)_{(v, w)} \hookrightarrow L(X, Y)$. By Lemma \ref{lemma: null homotopy}, this inclusion is nullhomotpic for all $(v, w)$. Commutativity of the above diagram then implies that the top-horizontal map is the zero map. However, as established before, this top-horizontal map is a surjection and so the fact that it is also the zero map implies that the group $\pi_{d}(F_{0})$ is the zero group. This completes the induction and the proof of the theorem.
\end{proof}

We have now established that if $(X, Y, \lambda, q, \alpha)$ is a Wall pairing of rank $g$ then $|L(X, Y)|$ is $\frac{1}{2}(g-3)$-connected. This implies that the geometric realization of the simplicial complex $K^{\pi}(W_{g,1})$ is $\frac{1}{2}(g-3)$-connected as well. In order to finish the proof of Theorem \ref{thm: pi cohen mac}, we need to verify that $K^{\pi}(W_{g,1})$ satisfies the weak Cohen-Macauley condition from Definition \ref{defn: cohen mac}. 

\begin{corollary}  The complex $K^{\pi}(W_{g,1})$ is \textit{weakly Cohen-Macaulay} of dimension $\frac{1}{2}(g-1)$. \end{corollary}
  \begin{proof} It will suffice to prove that for any $p$-simplex $v = ((v_{0}, w_{0}), \dots, (v_{p}, w_{p}))$, the complex, $\text{link}(v)$ is $(\frac{1}{2}(g - 1) - p - 1)$-connected. Let $V = \langle v_{0}, \dots, v_{p} \rangle \subset \pi_{n}(W_{g,1})$ and $W = \langle w_{0}, \dots, w_{p} \rangle \subset \pi_{n+1}(W_{g,1})$. Lemma \ref{lemma: compliment module} implies that the tuple,
  $$(W^{\perp}, V^{\perp}\cap \widehat{W}, \lambda'_{n, n+1}, \lambda_{n+1, n+1}', \alpha')$$ 
  is a Wall pairing of rank $g - p - 1$, where $\lambda'_{n, n+1}, \lambda_{n+1, n+1}',$ and $\alpha'$ denote the restrictions. Now, a simplex $\{(x_{0}, y_{0}), \dots, (x_{k}, y_{k})\} \in K^{\pi}(W_{g,1})$ is in the sub-complex $\text{link}(v)$ if and only if 
  $$\lambda_{n, n+1}(v_{i}, y_{j}) = 0, \quad  \lambda_{n, n+1}(x_{i}, w_{j}) = 0, \quad \text{and} \quad \lambda_{n+1, n+1}(y_{i}, w_{i}) = 0$$
  for all $i, j$. This implies that $\{(x_{0}, y_{0}), \dots, (x_{k}, y_{k})\}$ is in $\text{link}(v)$ if and only if  the sequence $((x_{0}, y_{0}), \dots, (x_{k}, y_{k}))$ (for any ordering of the vertices) it is an element of the poset $L(W^{\perp}, V^{\perp}\cap \widehat{W})$. From the discussion in the beginning of Section \ref{section: The Intersection Complex}, it follows that there is a homeomorphism of geometric realizations,
  $$|\text{link}(v)_{\bullet}| \; \cong \; |L(W^{\perp}, V^{\perp}\cap \widehat{W})|$$    
  where $\text{link}(v)_{\bullet}$ is the semi-simplicial set associated to the complex $\text{link}(v)$.
  The proof of the corollary then follows from Theorem \ref{theorem: wall complex connectivity}.
  \end{proof}
\section{Resolutions of Moduli Spaces} \label{Resolutions of Moduli Spaces}
\subsection{An Augmented Semi-Simplicial Space}
We now use the high-connectivity of the geometric realization
$|\bar{K}_{\bullet}(W_{g,1})|$ (Corollary \ref{connectivity of K-
  bar}) to prove Theorem \ref{thm: Main Theorem}. The constructions and
proofs of this section go through in exactly same way as in
\cite[Section 5]{GRW 12} and thus are provided for the convenience of the reader. 

Recall from Section \ref{Basic Definitions} that the topology of $\mathcal{M}_{g}$ was defined by the homeomorphism
$$\mathcal{M}_{g} \; \cong \; \Emb(W_{g,1}, \; [0,\infty)\times \R^{\infty})^{\partial}/\Diff(W_{g,1})^{\partial},$$
and that the quotient map 
$\Emb(W_{g,1}, \; [0,\infty)\times \R^{\infty})^{\partial} \longrightarrow \mathcal{M}_{g}$
is a locally trivial fibre bundle.

For each non-negative integer $p$, the topological group $\Diff(W_{g,1})^{\partial}$ acts on the space $\bar{K}_{p}(W_{g,1})$ by
$$(\psi, \; (\phi_{0}, t_{0}), \dots, (\phi_{p}, t_{p})) \; \mapsto \; ((\psi\circ\phi_{0}, t_{0}), \dots, (\psi\circ\phi_{p}, t_{p})),$$
where $\psi \in \Diff(W_{g,1})^{\partial}$ and $(\phi_{0}, t_{0}), \dots, (\phi_{p}, t_{p}) \in \bar{K}_{p}(W_{g,1})$. Notice that since $\Diff(W_{g,1})^{\partial}$ is defined to be the group of diffeomorphisms of $W_{g,1}$ that fix a neighborhood of the boundary $\partial W_{g,1}$ pointwise, the pair $(\psi\circ \phi_{i}, t_{i})$ still satisfies condition (i.) of Definition \ref{defn: the embedding complex} for $i = 0, \dots, p$. Thus, this group-action is well defined. 

\begin{defn} For each integer $p \geq 0$, we define $X_{p}(W_{g,1})$ to be the space of pairs $(W, \phi)$ where $W \in \mathcal{M}_{g}$ and $\phi \in \bar{K}_{p}(W_{g,1})$, topologized as 
$$X_{p}(W_{g,1}) = (\Emb(W_{g,1}, \; [0,\infty)\times \R^{\infty})^{\partial} \times \bar{K}_{p}(W_{g,1}))/ \Diff(W_{g,1})^{\partial}.$$
\end{defn}

This makes $X_{\bullet}$ into a semi-simplicial space augmented over $\mathcal{M}_{g}$. By the local triviality of the quotient map defining the topology on $\mathcal{M}_{g}$, the augmentation $X_{\bullet} \rightarrow \mathcal{M}_{g}$ is locally trivial with fibres $\bar{K}_{\bullet}(W_{g,1})$. 
\begin{proposition} The map $|X_{\bullet}(W_{g,1})| \rightarrow \mathcal{M}_{g}$ induced by the augmentation is $\frac{1}{2}(g-3)$-connected. \end{proposition} 
\begin{proof} This is the same as \cite[Proposition 5.2]{GRW 12}. The map $|X_{\bullet}| \rightarrow \mathcal{M}_{g}$ is a locally trivial fibre-bundle with fibre $|\bar{K}_{\bullet}(W_{g,1})|$ which is $\frac{1}{2}(g-3)$-connected. The claim follows from the long exact sequence in homotopy groups. \end{proof}

We now proceed to construct weak homotopy equivalences $\mathcal{M}_{g-p} \sim X_{p-1}$. Consider $S^{2n}$ as the submanifold of $\R^{\infty}$ given by
$$\{(x_{1}, \dots, x_{k}, \dots )) \in \R^{\infty} \; | \; x_{1}^{2} + \cdots + x_{2n+1}^{2} = 1; \; x_{j} = 0 \; \text{ if \; $j > 2n+1$}\}$$
and denote by $C$ the submanifold of $[0,1]\times\R^{\infty}$ given by the product,
$$[0,1]\times S^{2n} \subset [0,1]\times\R^{\infty}.$$  
Let $a: [0,1)\times \R^{2n} \rightarrow C$ be an embedding satisfying $a^{-1}(\{0\}\times S^{2n}) = \{0\}\times\R^{2n}$.
Let $S_{1} \subset [0, 1]\times \R^{\infty}$ be the manifold obtained from the cylinder $C$ by forming the connected sum with an embedded copy of $S^{n}\times S^{n+1}$ in $[0,1]\times\R^{\infty}$, along a small disk in $C$ disjoint from the image of the embedding $a$. From the connected-sum, $S_{1}$ comes with a canonical embedding $W_{1, 1} \hookrightarrow S_{1}$, and we pick an extension of this to an embedding $\phi_{0}: H \rightarrow S_{1}$, in such a way so that $(\phi_{0}, t_{0}) \in \bar{K}_{0}(S_{1}, a)$ for some choice of $t_{0} \in \R$, 
where $a$ is the embedding chosen above. 
(From here forward we will again suppress the embedding $a$ from the notation and write $\bar{K}_{0}(S_{1}) := \bar{K}_{0}(S_{1}, a)$.)
For $j$ a positive integer, let 
$T_{j}: \R\times\R^{\infty} \rightarrow \R\times\R^{\infty}$
be the linear translation given by the formula $(t, x) \; \mapsto \; (t + j, x)$. 
For positive integer $p$, we then define $S_{p} \subset [0, p]\times \R^{\infty}$ to be the submanifold given by the $p$-fold concatenation,
$$S_{p} := S_{1}\cup T_{1}(S_{1}) \cup \cdots \cup T_{p-1}(S_{1}).$$ 
This submanifold $S_{p}$ has $p$ pairwise disjoint embeddings 
$W_{1, 1} \rightarrow T_{j}(S_{1}) \hookrightarrow S_{p}$
coming from the connected sums of the copies of $S^{n}\times S^{n+1}$.  As above, we extend each of these to embeddings $\phi_{i}: H \rightarrow S_{p}$ so that 
$((\phi_{0}, t_{0}) \dots, (\phi_{p-1}, t_{p-1})) \in \bar{K}_{p-1}(S_{p})$
for some choice of real numbers $t_{0} < \cdots < t_{p-1}$. For any such choices we get a map
\begin{equation} \label{eq: stabilizer} 
\begin{aligned}
&\mathcal{M}_{g-p} \; \longrightarrow \; X_{p-1}, \\
& W \; \mapsto \; (S_{p} \cup (pe_{1} + W), \; ((\phi_{0}, t_{0}),  \dots, (\phi_{p-1}, t_{p-1})) \;). 
\end{aligned}
\end{equation}
\begin{proposition} \label{prop: stabilizer} For $g - p \geq 3$, the map $(\ref{eq: stabilizer})$ is a weak homotopy equivalence. \end{proposition} 
\begin{proof} This is the same as \cite[Proposition 5.3]{GRW 12}. The proof of Proposition 5.3 in \cite{GRW 12} uses analogues of Corollaries \ref{cor: transitivity} and \ref{cor: cancellation} for the manifolds considered in that paper. Since we have these two results, the proof goes through in exactly the same way.
\end{proof}

Recall that the stabilization map $s_{g-1}: \mathcal{M}_{g-1} \rightarrow \mathcal{M}_{g}$ was defined by a submanifold of $[0,1]\times\R^{\infty}$ diffeomorphic to $S^{n}\times S^{n+1}$ with the interior of two disks cut out. The space of such sub-manifolds is path connected and so the homotopy class of this stabilization map does not depend on the choice of sub-manifold of $[0,1]\times\R^{\infty}$. To be precise, we use the the submanifold $S_{1} \subset [0,1]\times \R^{\infty}$ to define our stabilization map. Together with Proposition \ref{prop: stabilizer}, our next result says that the last face map of $X_{\bullet}$ is a model for the stabilization map. 

\begin{proposition} \label{prop: stabilization diff}  The following diagram is commutative for $p \geq 0$
$$\xymatrix{
\mathcal{M}_{g-p-1} \ar[rr]^{s_{g-p-1}} \ar[d] && \mathcal{M}_{g-p} \ar[d] \\
X_{p} \ar[rr]^{d_{p}} && X_{p-1}}$$
where the vertical maps are the given by $(\ref{eq: stabilizer})$ and the top horizontal map is the stabilization map from $(\ref{eq: stab map})$. \end{proposition}

\begin{proof} Starting with $W \in \mathcal{M}_{g-p-1}$, we have $s_{g-p-1}(W) = S_{1} \cup (e_{1} + W) \in \mathcal{M}_{g-p}$. The image of the $S_{1} \cup (e_{1} + W)$ under the right-vertical map is 
$$S_{p}\cup(pe_{1} + (S_{1} \cup (e_{1} + W))) \; = \; S_{p+1}\cup ((p+1)e_{1} + W) \subset [0, p+1]\times\R^{\infty}$$
equipped with the embeddings $(\phi_{0}, \dots, \phi_{p-1})$. If instead we map $W$ down to $W_{p}$, we get the element with the same underlying manifold but equipped with the embeddings $(\phi_{0}, \dots, \phi_{p})$, and the face map $d_{p}: X_{p} \rightarrow X_{p-1}$ then forgets the embedding $\phi_{p}$.  
\end{proof}

Finally, we show that all face maps $d_{i}: X_{p} \rightarrow X_{p-1}$ are homotopic, $i = 1, \dots, p$. For this, we need to be more precise about our choices of 
$((\phi_{0}, t_{0}), \dots, (\phi_{p-1}, t_{p})) \in \bar{K}_{p-1}(S_{p})$
used in the map from (\ref{eq: stabilizer}). First, the inclusion $S_{p} \hookrightarrow S_{p+1}$ induces a map $\bar{K}_{\bullet}(S_{p}) \rightarrow \bar{K}_{\bullet}(S_{p+1})$ which we may assume sends the $(\phi_{i}, t_{i})  \in \bar{K}_{0}(S_{p})$ to the vertices in $\bar{K}_{0}(S_{p+1})$ with the same names. By the proof of Corollary \ref{cor: transitivity}, we may pick a diffeomorphism $\psi: S_{2} \stackrel{\cong} \longrightarrow S_{2}$ which interchanges the two canonical embeddings $W_{1,1} \hookrightarrow S_{2}$ while fixing pointwise the image  
$a([0,1)\times \R^{2n}) \subset S_{2}$, and a neighborhood of the boundary $\partial S_{2}$ (recall $a$ was the coordinate patch used in the definition of $\bar{K}_{\bullet}(S_{2})$).   
We now pick $(\phi_{1}, t_{1}) \in \bar{K}_{0}(S_{2})$ so that $\phi_{1}$ is in the same path component as $\psi\circ \phi_{0}$. More generally for $p \geq 2$ and $1 \leq i < p$, we let $\psi_{(i-1, i)}$ be the diffeomorphism of $S_{p} \subset [0,p]\times \R^{\infty}$ which acts as $\psi$ inside $S_{p}\cap ([i-1, i+1]\times\R^{\infty})$ and is the identity outside. We may then inductively pick $(\phi_{p}, t_{p}) \in \bar{K}_{0}(S_{p+1})$ so that $\phi_{p}$ is in the same path component as $\psi_{(p-1, p)}\circ \phi_{p-1}$ and such that $\phi_{p}$ is disjoint from the images of $\phi_{i}$ and the support of $\psi_{(i-1, i)}$ for $i < p$.

\begin{proposition} \label{prop: differentials} For $0 \leq p \leq g-3$, all face maps $d_{i}: X_{p} \rightarrow X_{p-1}$ are weakly homotopic to each other. \end{proposition}
 \begin{proof} This is the same as the proof of \cite[Proposition 5.5]{GRW 12}. We focus on the case $p=1$, the general case is similar. For $i = 0,1$, we denote by $f_{i}$ the composition of $d_{i}: X_{1} \rightarrow X_{0}$ with the weak equivalence $\mathcal{M}_{g-2} \rightarrow X_{1}$ from Proposition \ref{prop: stabilizer}. We shall construct a homotopy $f_{0}\sim f_{1}: \mathcal{M}_{g-2}\rightarrow X_{0}$. These maps are given by the formula 
 $$f_{i}(W) \; = \; (S_{2}\cup(2e_{1} + W), (\phi_{i}, t_{i})).$$
 The composition of the inclusion $i: S_{2} \hookrightarrow [0,2]\times\R^{\infty}$ with the diffeomorphism $\psi: S_{2} \stackrel{\cong}\longrightarrow S_{2}$ is an embedding which agrees with $i$ near $\partial S_{2}$. The space of all embeddings $q: S_{2} \rightarrow [0,2]\times\R^{\infty}$ which satisfy $q|_{\partial S_{2}} = i|_{\partial S_{2}}$ is path connected. We therefore may choose an isotopy 
 $$h_{t}: S_{2} \longrightarrow [0,2]\times\R^{\infty}$$
 from $i$ to $i\circ\psi$, which restricts to the constant isotopy of embeddings on a neighborhood of $\partial S_{2}$. The formula
 $$W \; \mapsto \; (h_{t}(S_{2})\cup(2e_{1} + W), (h_{t}\circ\phi_{1}, t_{1}))$$
 gives a homotopy of maps $\mathcal{M}_{g-2}\rightarrow X_{0}$ which starts at $f_{1}$ and ends at the map $h_{1}: W \mapsto (S_{2}\cup(2e_{1} + W), (\psi\circ\phi_{1}, t_{1}))$. Since $\psi\circ \phi_{1}$ is in the same path component as $\phi_{0}$, the map $h_{1}$ is homotopic to $f_{0}$. 
 \end{proof}

\subsection{The Spectral Sequence and Homological Stability} \label{The Spectral Sequence and Homological Stability}
We now prove Theorem \ref{thm: Main Theorem} by induction. We consider the spectral sequence induced by the augmented semi-simplicial space $X_{\bullet}$, with $E^{1}$ term given by $E^{1}_{p,q} \; = \; H_{q}(X_{p})$ for $p \geq -1$ and $q \geq 0$. The differential is given by $d^{1} = \sum(-1)^{i}(d_{i})_{*}$, and the group $E^{\infty}_{p,q}$ is a sub-quotient of the relative homology $H_{p+q+1}(X_{-1}, |X_{\bullet}|)$. The proofs in this section are exactly the same as \cite[Section 5.2]{GRW 12}. We repeat them here for the sake of completeness. 

\begin{lemma} \label{lemma: differentials} We have isomorphisms $E^{1}_{p,q} \cong H_{q}(\mathcal{M}_{g-p-1})$ for $ -1 \leq p \leq g-3$, with respect to which the differential 
$$H_{q}(\mathcal{M}_{g-p-1}) \cong E^{1}_{p,q} \stackrel{d^{1}} \longrightarrow E^{1}_{p,q} \cong H_{q}(\mathcal{M}_{g-p})$$
agrees with the stabilization map for $p$ even, and is zero otherwise. Furthermore, $E^{\infty}_{p,q} = 0$ for $p+q \leq \frac{1}{2}(g-3)$. 
\end{lemma}
\begin{proof} Proposition \ref{prop: stabilizer} identifies $E^{1}_{p,q} = H_{q}(X_{p}) \cong H_{q}(\mathcal{M}_{g-p-1})$ and Proposition \ref{prop: differentials} shows that all maps $(d_{i})_{*}: H_{q}(X_{p}) \longrightarrow H_{q}(X_{p-1})$ are equal, $i = 0, \dots, p$. Therefore all terms in the differential $d^{1} = \sum(-1)^{i}(d_{i})_{*}$ cancel for $p$ odd, and for $p$ even $(d_{p})_{*}$ survives and by Proposition \ref{prop: stabilization diff} is identified with the stabilization map. 

The group $E^{\infty}_{p, q}$ is a subquptient of the relative homology $H_{p+q+1}(X_{-1}, |X_{\bullet}|)$, but this vanishes for $p+q+1 \leq \frac{1}{2}(g-3)$ since the map $|X_{\bullet}| \longrightarrow X_{-1}$ is $\frac{1}{2}(g-3)$-connected. 
\end{proof}

\begin{proof}[Proof of Theorem \ref{thm: Main Theorem}] Let us write $a = \frac{1}{2}(g-3)$. We will use the spectral sequence above to prove that $H_{q}(\mathcal{M}_{g-1}) \longrightarrow H_{q}(\mathcal{M}_{g})$ is an isomorphism for $q \leq a$, assuming that we know inductively that for $j > 0$ the stabilization maps $H_{q}(\mathcal{M}_{g-2j-1}) \longrightarrow H_{q}(\mathcal{M}_{g- 2j})$ are isomorphisms for $q \leq a - j$. By Lemma \ref{lemma: differentials}, this implies that the differential $d^{1}: E^{1}_{2j, q} \longrightarrow E^{1}_{2j-1, q}$ is an isomorphism for $0 < p \leq 2(a - q)$. In particular, the $E^{2}_{p, q}$-term vanishes in the region given by $p \geq 1, q \leq a -1$ and $p + q \leq a + 1$, and thus for $r \geq 2$ and $q \leq a$ it follows that differentials into $E^{r}_{-1, q}$ and $E^{r}_{0, q}$ vanish. We deduce that for $q \leq a$ we have 
$$\begin{aligned}
E^{\infty}_{0, q} = E^{2}_{0, q} &= \Ker(H_{q}(\mathcal{M}_{g-1}) \rightarrow H_{q}(\mathcal{M}_{g}))\\
E^{\infty}_{-1, q} = E^{2}_{-1, q} & = \Coker(H_{q}(\mathcal{M}_{g-1}) \rightarrow H_{q}(\mathcal{M}_{g})),
\end{aligned}$$
and since the group $E^{\infty}_{p, q}$ vanishes for $p + q \leq a$ we
see that the stabilization map $$H_{q}(\mathcal{M}_{g-1}) \rightarrow
H_{q}(\mathcal{M}_{g})$$ has vanishing kernel and cokernel for $q \leq
a$, establishing the induction step. The statement is vacuous for $g =
1$ and $g = 2$, which starts the induction.
\end{proof}


  \end{document}